\documentclass[a4paper,11pt]{article}
\usepackage[top=2.5cm, bottom=2.6cm, left=3cm, right=2.6cm]{geometry}
\usepackage{amsfonts}
\usepackage{mathrsfs}
\usepackage{amsmath}
\usepackage{latexsym}
\usepackage{amssymb}
\usepackage{amsthm}
\usepackage{latexsym}
\usepackage{indentfirst}

\newtheorem{theorem}{Theorem}[section]

\newtheorem{lemma}{Lemma}[section]

\begin{document}

\title{Canonical bases of the modified quantized enveloping algebras of type $A_{2}$}
\author{Weideng Cui \\
 {\small Mathematical Sciences Center, Tsinghua University, Jin Chun Yuan West Building,}\\
{\small Beijing, 100084, P. R. China. E-mail: cwdeng@amss.ac.cn
 }}
\date{}
\maketitle
 \begin{abstract}The modified quantized enveloping algebra $\dot{\mathbf{U}}$ has a remarkable canonical basis, which was introduced by Lusztig. In this paper, we give an explicit description of all elements of
 the canonical basis of $\dot{\mathbf{U}}$ for type $A_{2}$.\end{abstract}

 \large

\section{Introduction}
\medskip
The canonical basis of a quantized enveloping algebra was introduced by
Lusztig first for type ADE in [L1] and then for other types (see
[L3, L5]). Kashiwara constructed the crystal basis and the global crystal
basis for the quantized enveloping algebra associated to an
arbitrary symmetrizable generalized Cartan matrix in [K1]. The
canonical basis and the global crystal basis of a quantized
enveloping algebra were proved to be the same by Lusztig for type
ADE in [L2] and by Grojnowski-Lusztig for symmetric generalized
Cartan matrices in [GL]. The canonical basis and the crystal basis have many
remarkable properties.

However it is hard to compute the canonical basis. By now, the basis
is only computed for type $A_{1}$, $A_2$, $A_3$, $B_2$ (see
[L1, X1, X2]). For type $A_4$, part of the basis is computed in
[HYY, HY, LH].

A modified form of the quantized enveloping algebra was introduced
by Lusztig and a remarkable basis was also given in [L4, L5], which
was called the canonical basis. It is even harder to compute the basis.
By now, the basis is only computed for type $A_{1}$ in Lusztig's
book [L5], which consists of two classes of monomial elements (see [L5, Prop 25.3.2]).

In this article, we try to compute the canonical basis of the
modified quantized enveloping algebra of type $A_2$. We give a full
list of all elements of the canonical basis in Theorem 3.1.

The contents of the article are as follows. We will consider the
quantized enveloping algebra of type $A_2$ and its modified form, we
denote them by $\mathbf{U},\dot{\mathbf{U}}$ respectively.

In Section 2, we recall some facts about $\mathbf{U}, \dot{\mathbf{U}}$
and the canonical basis $\dot{\mathbf{B}}$ of $\dot{\mathbf{U}}$ ,
the basic references are [L1, L5]. In Section 3, we give a description of all elements of the canonical
basis $\dot{\mathbf{B}}$ in Theorem 3.1. In Section 4, we first establish some combinatorial identities in Lemma 4.1, and then give the proof of Theorem 3.1.

We hope that the results of this article will be helpful to
understand the canonical basis of the modified quantized enveloping
algebra.

\section{Preliminaries}
\subsection{The modified quantized enveloping algebras of type $A_{2}$}We will need some notations.
Let $a$ be an integer and $b$ a positive number. Set \\
$$[a]=\frac{v^{a}-v^{-a}}{v-v^{-1}},~~~[b]!=\prod_{h=1}^{b}\frac{v^{h}-v^{-h}}{v-v^{-1}},~~~[0]!=1;$$
$$[-b]!=(-1)^{b}[b]!,~~~\left[\begin{array}{c}a \\
b
\end{array}\right]=\prod_{h=1}^{b}\frac{v^{a-h+1}-v^{-(a-h+1)}}{v^{h}-v^{-h}};$$
$$\left[\begin{array}{c}a \\
0
\end{array}\right]=1,~~~ \left[\begin{array}{c}a \\
-b
\end{array}\right]=0.$$

We have\\$$\left[\begin{array}{c}a \\
t
\end{array}\right]=0 ~~\mbox{if}~ 0\leq a< t,~~~\left[\begin{array}{c}a \\
t
\end{array}\right]=(-1)^{t}\left[\begin{array}{c}-a+t-1 \\
t
\end{array}\right];$$

$$\left[\begin{array}{c}a+b \\
b
\end{array}\right]=\frac{[a+b]!}{[a]![b]!} \   \hbox {  for } a,b\in \mathbb{N}.$$

$\mathbf{U}=U_{v}(sl_{3})$ is the associative algebra over
$\mathbb{Q}(v)$ ($v$ an indeterminate) with the generators
$e_{i},f_{i},k_{i}^{\pm 1}$ $(1\leq i \leq 2)$ subject to the following
relations:
\begin{equation*}k_{i}k_{j}=k_{j}k_{i}, ~~k_{i}k_{i}^{-1}=1,\
~~\forall 1\leq i,j \leq 2;\end{equation*}
\begin{equation*}k_{i}e_{j}k_{i}^{-1}=v^{\langle \alpha_{i}^{\vee},\alpha_{j}\rangle}e_{j},~~~k_{i}f_{j}k_{i}^{-1}=v^{-\langle\alpha_{i}^{\vee},\alpha_{j}\rangle}f_{j},
\ ~~\forall 1\leq i,j\leq 2;\end{equation*}
\begin{equation*}e_{i}f_{j}-f_{j}e_{i}=\delta_{i,j}\frac{k_{i}-k_{i}^{-1}}{v-v^{-1}};\end{equation*}
\begin{equation*}\sum\limits_{r+s=2}(-1)^{r}e_{i}^{(r)}e_{j}e_{i}^{(s)}=0,\  ~~\hbox{for }i\neq j;\end{equation*}
\begin{equation*}\sum\limits_{r+s=2}(-1)^{r}f_{i}^{(r)}f_{j}f_{i}^{(s)}=0,\ ~~ \hbox{for } i\neq j.\end{equation*}
where $e_{i}^{(r)}$ denotes the divided power
$\frac{e_{i}^{r}}{[r]!}$, and $k_{i}=k_{\alpha_{i}^{\vee}},$
$\alpha_{i}^{\vee}$ is the coroot corresponding to the simple root
$\alpha_{i}.$ We denote by $X,Y$ the weight lattice and coroot lattice
respectively, $\langle\cdot, \cdot\rangle:Y\times X\rightarrow
\mathbb{Z}$ is the perfect bilinear pairing. Let
$a_{ij}=\langle\alpha_{i}^{\vee},\alpha_{j}\rangle,$ then $(a_{ij})$
is the corresponding Cartan matrix.

There is a unique isomorphism of $\mathbb{Q}(v)$-vector spaces
$\sigma:\mathbf{U}\rightarrow
  \mathbf{U}$
such that
$$\sigma(e_{i})=e_{i},\sigma(f_{i})=f_{i},\sigma(k_{i})=k_{i}^{-1}
~\mbox{for}~ 1\leq i\leq 2,~~ \sigma(uu')=\sigma(u')\sigma(u)~ \mbox{for}~
u,u'\in \mathbf{U}.$$

Let
$-:\mathbf{U}\to \mathbf{U}$ be the bar involution of $\mathbf{U}$,
which is a ring homomorphism of $\mathbf{U}$ defined by
$$\bar e_i=e_i,~~~~ \bar f_i=f_i, ~~~~\bar k_i=k_i^{-1}~~ \mbox{for}~ 1\leq i \leq 2, ~~~~ \bar v=v^{-1}.$$
It is well-known that $\mathbf{U}$ is a Hopf algebra and the comultiplication
is defined as follows:
$$\Delta(e_{i})=e_{i}\otimes 1+k_{i}\otimes
e_{i},~~~\Delta(f_{i})=f_{i}\otimes k_{i}^{-1}+1\otimes
f_{i},~~~\Delta(k_{i})=k_{i}\otimes k_{i}~~ \mbox{for}~ 1\leq i \leq 2.$$

We define a category $\mathcal{C}$ as follows. An object of
   $\mathcal{C}$ is a $\mathbf{U}$-module $M$ with a given direct
   sum decomposition $M=\bigoplus_{\lambda\in X}M_{\lambda}$(as a $\mathbb{Q}(v)$-vector space)
   such that for any $\lambda\in X$ and $m\in
   M_{\lambda},$ we have $k_{i}m=v^{<\alpha_{i}^{\vee},\lambda>}m.$ The
   subspaces $M_{\lambda}$ are called the weight spaces of $M.$ A
   morphism in $\mathcal{C}$ is a $\mathbf{U}$-linear map.

If $M',M'' \in \mathcal{C},$ the tensor product $M'\otimes M''$ is
naturally a $\mathbf{U}\otimes \mathbf{U}$-module with $(u'\otimes
u'')(m'\otimes m'')=u'm'\otimes u''m''.$ We restrict it to a
$\mathbf{U}$-module via the algebra homomorphism
$\Delta:\mathbf{U}\rightarrow\mathbf{U}\otimes \mathbf{U}$. The
resulting $\mathbf{U}$-module is naturally an object of
$\mathcal{C}$. If $m'\in M'_{\lambda^{'}},m''\in
M''_{\lambda^{''}},$ we have the following identities:

\vskip3mm (a)\hspace*{0.5cm} $e_{i}^{(a)}(m'\otimes
m'')=\sum\limits_{a'+a''=a}v^{a'a''+a''<\alpha_{i}^{\vee},\lambda'>}e_{i}^{(a')}m'\otimes
e_{i}^{(a'')}m'';$

\vskip3mm (b) $\hspace*{0.5cm}f_{i}^{(a)}(m'\otimes
m'')=\sum\limits_{a'+a''=a}v^{a'a''-a'<\alpha_{i}^{\vee},\lambda''>}f_{i}^{(a')}m'\otimes
f_{i}^{(a'')}m''.$

\vskip3mm
Let $\mathbf{U}^{+}$ be the $\mathbb{Q}(v)$-subalgebra of
$\mathbf{U}$ generated by all the $e_{i}$ $(1\le i\le 2),$ and
let $\mathbf{U}^{-}$ be the $\mathbb{Q}(v)$-subalgebra of $\mathbf{U}$
generated by all the $f_{i}$ $(1\le i\le 2)$. It is
well-known that if $\mathbf{f}$ is the associative
$\mathbb{Q}(v)$-algebra generated by $\theta_{i}$ $(1\leq i \leq 2)$ subject to the following relation:
$$\sum\limits_{r+s=2}(-1)^{r}\theta_{i}^{(r)}\theta_{j}\theta_{i}^{(s)}=0,
\ ~~\hbox {for }i\neq j.$$
There is a unique algebra isomorphism $+:\mathbf{f}\rightarrow \mathbf{U}^{+}; b\mapsto b^{+}$ (resp. $-:\mathbf{f}\rightarrow \mathbf{U}^{-}; b\mapsto b^{-}$) such that $\theta_{i}^{+}=e_{i}$ (resp. $\theta_{i}^{-}=f_{i}$). Let
$\mathcal{A}=\mathbb{Z}[v,v^{-1}]$, and let
$\mathbf{f}_{\mathcal{A}}$ be the $\mathcal{A}$-subalgebra of
$\mathbf{f}$ generated by the elements $\theta_{i}^{(s)}$ for $1\leq i
\leq 2$ and $s\in \mathbb{N}$.

Let $I=\{1,2\},$ and let $\mathbb{N}[I]$ be the monoid consisting of all
linear combinations of elements of $I$ with coefficients in
$\mathbb{N}.$ For any $\nu=\sum_{i}\nu_{i}i\in \mathbb{N}[I],$ we
denote by $\mathbf{f}_{\nu}$ the $\mathbb{Q}(v)$-subspace of
$\mathbf{f}$ spanned by the monomials $\theta_{i_1}\cdots
\theta_{i_{r}}$ such that for any $i\in I,$ the number of
occurrences of $i$ in the sequence $i_1,\ldots, i_r$ is equal to
$\nu_{i}.$ Then each $\mathbf{f}_{\nu}$ is a finite dimensional
$\mathbb{Q}(v)$-vector space and we have a direct sum decomposition
$\mathbf{f}=\bigoplus_{\nu\in \mathbb{N}[I]}\mathbf{f}_{\nu}.$ We have
$\mathbf{f}_{\nu}\mathbf{f}_{\nu'}\subset\mathbf{f}_{\nu+\nu'}, 1\in
\mathbf{f}_{0}$ and $\theta_{i}\in \mathbf{f}_{i}.$ By the above-mentioned algebra
isomorphisms, we also have
$\mathbf{U}^{-}=\bigoplus_{\nu\in \mathbb{N}[I]}U^-_{-\nu}$ and
$\mathbf{U}^{+}=\bigoplus_{\nu\in \mathbb{N}[I]}U^+_{\nu}.$ Any
element $x$ of $\mathbf{f}$ is said to be homogeneous if it belongs
to $\mathbf{f}_{\nu}$ for some $\nu.$ We then set $|x|=\nu.$ If
$\nu=\sum_{i}\nu_{i}i\in \mathbb{N}[I],$ then we set
$tr\hspace*{0.1cm} \nu=\sum_{i}\nu_{i}.$

For an integer $c$ and a positive integer $a$ we set
$$\left[\begin{array}{c}k_{i};c \\ a \end{array}\right]
=\prod\limits_{h=1}^{a}\frac{k_{i}v^{c-h+1}-k_{i}^{-1}v^{-(c-h+1)}}{v^{h}-v^{-h}}
  ,  ~~\left[\begin{array}{c}k_{i};c \\ 0\end{array}\right]=1.$$
  We have the following formulas (see [L5]):

\vskip3mm (c) ${\displaystyle
e_{i}^{(a)}f_{i}^{(b)}=\sum\limits_{0\leq t\leq a,b}f_{i}^{(b-t)}
\left[\begin{array}{c}k_{i};2t-a-b \\ t
\end{array}\right]e_{i}^{(a-t)};}$

\vskip3mm

(d) $e_{i}^{(a)}f_{j}^{(b)}=f_{j}^{(b)}e_{i}^{(a)} ~~\hbox{if} \
i\neq j;$

\vskip3mm (e) $ \left[\begin{array}{c}k_{i};c
\\ a
\end{array}\right]\hspace*{-0.1cm}e_j^{(b)}=e_j^{(b)}\hspace*{-0.1cm}\left[\begin{array}{c}k_{i};c+ba_{ij} \\ a
\end{array}\right]\hspace*{-0.1cm},$~~~
${\displaystyle \left[\begin{array}{c}k_{i};c \\ a
\end{array}\right]\hspace*{-0.1cm}f_j^{(b)}=f_j^{(b)}\hspace*{-0.1cm}\left[\begin{array}{c}k_{i};c-ba_{ij} \\ a
\end{array}\right].}$\\

\vskip2mm Next let us recall the definition of the modified
quantized enveloping algebra $\dot{\mathbf{U}}$, which is the
modified form of $\mathbf{U}$.

 If $\lambda^{'},\lambda^{''}\in X,$ we set
$${}_{\lambda^{'}}\!\mathbf{U}_{\lambda^{''}}=\mathbf{U}/(\sum\limits_{\mu\in Y}(k_{\mu}-v^{<\mu,\lambda^{'}>})\mathbf{U}+
\sum\limits_{\mu\in Y}\mathbf{U}(k_{\mu}-v^{<\mu,\lambda^{''}>})).$$

Let $\pi_{\lambda',  \lambda''}:\mathbf{U}\rightarrow
 {}_{\lambda^{'}}\!\mathbf{U}_{\lambda^{''}}$ be the canonical projection. Then we
define $$\dot{\mathbf{U}}=\bigoplus_{\lambda^{'}, \lambda^{''}\in
X}{}_{\lambda^{'}}\!\mathbf{U}_{\lambda^{''}}.$$

There is a natural associative $\mathbb{Q}(v)$-algebra structure on
$\dot{\mathbf{U}}$ inherited from that of $\mathbf{U}.$ The elements
$1_{\lambda}=\pi_{\lambda,\lambda}(1)$ $(\lambda \in X)$ of
$\dot{\mathbf{U}}$ satisfy
$$1_{\lambda}1_{\lambda^{'}}=\delta_{\lambda,\lambda^{'}}1_{\lambda}.$$

Then we have
$${}_{\lambda^{'}}\!\mathbf{U}_{\lambda^{''}}=1_{\lambda^{'}}\dot{\mathbf{U}}1_{\lambda^{''}}.$$

The algebra $\dot{\mathbf{U}}$ does not generally have $1$, but
instead a collection of orthogonal idempotents.

We have the following identities in $\dot{\mathbf{U}}$:
$$e_{i}^{(a)}1_{\lambda}=1_{\lambda+a\alpha_{i}}e_{i}^{(a)},~~
f_{i}^{(b)}1_{\lambda}=1_{\lambda-b\alpha_{i}}f_{i}^{(b)}\hbox{ for
}
 1\leq i \leq 2,~~\lambda \in X,~~a,b\in \mathbb{N}.$$

If $\omega_{1},\omega_{2}$ are the fundamental weights in $X$, then
any element of $X$ can be written as
$\lambda=\lambda_{1}\omega_{1}+\lambda_{2}\omega_{2}$, for $\lambda_{1},\lambda_{2}\in
\mathbb{Z}.$ We will sometimes denote this element $\lambda$ by
$(\lambda_{1},\lambda_{2})$ for simplicity. So the above
identities can be written as follows:

\vskip3mm
(f)\hspace*{0.2cm}$e_{1}^{(a)}1_{(\lambda_{1},\lambda_{2})}=1_{(\lambda_{1}+2a,\lambda_{2}-a)}e_{1}^{(a)},~~~
e_{2}^{(a)}1_{(\lambda_{1},\lambda_{2})}=1_{(\lambda_{1}-a,\lambda_{2}+2a)}e_{2}^{(a)};$

\vskip3mm
(g)\hspace*{0.2cm}$f_{1}^{(b)}1_{(\lambda_{1},\lambda_{2})}=1_{(\lambda_{1}-2b,\lambda_{2}+b)}f_{1}^{(b)},~~~
f_{2}^{(b)}1_{(\lambda_{1},\lambda_{2})}=1_{(\lambda_{1}+b,\lambda_{2}-2b)}f_{2}^{(b)}.$

\vskip3mm The $\mathbb{Q}$-algebra homomorphism
$-:\mathbf{U}\rightarrow
  \mathbf{U}$ induces, for any $\lambda,\lambda'\in X,$ a
  $\mathbb{Q}$-linear map $-:{}_{\lambda^{'}}\!\mathbf{U}_{\lambda^{''}}\rightarrow {}_{\lambda^{'}}\!\mathbf{U}_{\lambda^{''}}.$
Taking the direct sum of these maps, we obtain a $\mathbb{Q}$-linear
map $-:\dot{\mathbf{U}}\rightarrow
  \dot{\mathbf{U}}$ with square $1,$ which respects the
  multiplication of $\dot{\mathbf{U}},$ maps each $1_{\lambda}$ into
  itself and satisfies $\overline{tst'}=\bar{t}\bar{s}\bar{t'}$ for
  $t,t'\in \mathbf{U}$
  and $s\in \dot{\mathbf{U}}.$

The map $\sigma:\mathbf{U}\rightarrow
  \mathbf{U}$ induces, for each $\lambda',\lambda'',$ a linear
  isomorphism ${}_{\lambda^{'}}\!\mathbf{U}_{\lambda^{''}}\rightarrow
  {}_{-\lambda^{''}}\!\mathbf{U}_{-\lambda^{'}}.$ Taking direct sums, we obtain a linear
  isomorphism $\sigma:\dot{\mathbf{U}}\rightarrow
  \dot{\mathbf{U}}$ such that $\sigma(1_{\lambda})=1_{-\lambda}$ for
  all $\lambda \in X,$ and
  $\sigma(uxx'u')=\sigma(u')\sigma(x')\sigma(x)\sigma(u)$ for all $u,u'\in
  \mathbf{U}$ and $x,x'\in \dot{\mathbf{U}}.$

The $\mathcal{A}$-submodule of $\dot{\mathbf{U}}$ spanned by the
elements $x^{+}1_{\lambda}x'^{-}$(with $x,x'\in
\mathbf{f}_{\mathcal{A}},\lambda\in X)$ coincides with the
$\mathcal{A}$-submodule of $\dot{\mathbf{U}}$ spanned by the
elements $x^{-}1_{\lambda}x'^{+}$(with $x,x'\in
\mathbf{f}_{\mathcal{A}},\lambda\in X),$ we denote it by
$\dot{\mathbf{U}}_{\mathcal{A}}$

\subsection{Canonical bases of modified quantized enveloping algebras of type $A_{2}$}

Assume that $V(a\omega_{1}+b\omega_{2})$ is a finite dimensional simple
highest weight $\mathbf{U}$-module with the highest weight vector
$\eta_{(a,b)}$ $(a, b\in \mathbb{N})$, and that $V(-s\omega_{1}-t\omega_{2})$ is a finite dimensional simple lowest weight $\mathbf{U}$-module with the
lowest weight vector $\xi_{(-s,-t)}$ $(s, t\in \mathbb{N})$. $V(a\omega_{1}+b\omega_{2})$ and $V(-s\omega_{1}-t\omega_{2})$ are
both objects of $\mathcal{C}$.

As is well known, the canonical basis $\mathbf{B}$ of $\mathbf{f}$
is given by Lusztig in [L1] as follows:
$$\theta_{2}^{(u)}\theta_{1}^{(v)}\theta_{2}^{(w)},\  \theta_{1}^{(u)}\theta_{2}^{(v)}\theta_{1}^{(w)},~~~\mbox{if}~v\geq u+w ~\mbox {and} ~u,v,w \in \mathbb{N}.$$
where
$\theta_{2}^{(u)}\theta_{1}^{(u+w)}\theta_{2}^{(w)}=\theta_{1}^{(w)}\theta_{2}^{(u+w)}\theta_{1}^{(u)}$
are considered only once.

In general, given $\lambda\in X^{+},$ we define
$\mathbf{B}(\lambda)$ as in [L5, Theorem 14.4.11], and the map
$b\mapsto b^{-}\eta_{\lambda}$ defines a bijection of
$\mathbf{B}(\lambda)$ onto a basis of the highest weight
$\mathbf{U}$-module $V(\lambda)$ with the highest weight vector
$\eta_{\lambda};$ $b^{-}\eta_{\lambda}=0$ if $b\notin
\mathbf{B}(\lambda).$ Making use of [L5, Theorem
14.4.11], we have the following lemma.

\begin{lemma}If $\lambda=a\omega_{1}+b\omega_2, a,b\in\mathbb{N},$
we can get that $\mathbf{B}(a\omega_{1}+b\omega_{2})$ consists of
the following elements:
$$\theta_{2}^{(u)}\theta_{1}^{(v)}\theta_{2}^{(w)},\ \hbox{ where }0\leq w \leq b,~~0\leq u \leq a ,~~u+w\leq v \leq a+w; \eqno{(1)}$$
$$\theta_{1}^{(s)}\theta_{2}^{(t)}\theta_{1}^{(r)},\ \hbox{where }0\leq s \leq b-1,~~0\leq r\leq a, ~~s+1+r\leq t \leq b+r. \eqno{(2)}$$

If $b=0$, we only consider these elements listed in $(1)$.\end{lemma}
\vskip2mm Given
$(s,t),(a,b)\in X^{+},$ we will consider the following partial order
on the set
$\mathbf{B}(s\omega_{1}+t\omega_{2})\times\mathbf{B}(a\omega_{1}+b\omega_{2}).$
We say that $(b_{1},b_1')\leq (b_2,b_2')$ if
$tr\hspace*{0.1cm}|b_1|-tr\hspace*{0.1cm}|b_1'|=tr\hspace*{0.1cm}|b_2|-tr\hspace*{0.1cm}|b_2'|,$
and if we have either $tr\hspace*{0.1cm}|b_1|< tr
\hspace*{0.1cm}|b_2|$ and  $tr \hspace*{0.1cm}|b_1'|< tr
\hspace*{0.1cm}|b_2'|,$ or $b_1=b_2,b_1'=b_2'.$

Let $-:V(a\omega_{1}+b\omega_{2})\rightarrow
V(a\omega_{1}+b\omega_{2})$ be the unique $\mathbb{Q}$-linear
involution such that
$\overline{u\eta_{(a,b)}}=\bar{u}\eta_{(a,b)},\forall u\in
    \mathbf{U};$ similarly, let $-:V(-s\omega_{1}-t\omega_{2})\rightarrow
    V(-s\omega_{1}-t\omega_{2})$ be the unique $\mathbb{Q}$-linear
involution such that
$\overline{u\xi_{(-s,-t)}}=\bar{u}\xi_{(-s,-t)},\forall~ u\in
\mathbf{U}.$ Let
$-\hspace*{-0.05cm}=\hspace*{-0.05cm}-\otimes-\hspace*{-0.05cm}:\hspace*{-0.05cm}V(-s\omega_{1}-t\omega_{2})\otimes
V(a\omega_{1}+b\omega_{2})\rightarrow
V(-s\omega_{1}-t\omega_{2})\otimes V(a\omega_{1}+b\omega_{2}).$
These elements $b^{+}\xi_{(-s,-t)}\otimes b'^{-}\eta_{(a,b)}$ with
$b\in \mathbf{B}(s\omega_{1}+t\omega_{2}),b'\in
\mathbf{B}(a\omega_{1}+b\omega_{2})$ form a $\mathbb{Q}(v)$-basis of
$V(-s\omega_{1}-t\omega_{2})\otimes V(a\omega_{1}+b\omega_{2}).$
They generate a $\mathbb{Z}[v^{-1}]$-submodule $\mathcal{L}$ and an
$\mathcal{A}$-submodule $\mathcal{L}_{\mathcal{A}}.$

By [L5, Theorem 4.1.2], there is a unique family of elements
$\Theta_{\nu}\in U^-_{-\nu}\otimes U^+_{\nu}(\nu\in \mathbb{N}[I])$
such that $\Theta_{0}=1\otimes1$ and
$\Theta=\sum\limits_{\nu}\Theta_{\nu}$ satisfies
$\Delta(u)\Theta=\Theta\overline{\Delta(\bar{u})},\forall ~u\in
\mathbf{U}.$ The element
$\Theta$ is called the quasi-R-matrix. Then we can define a linear map
$\Theta^{\prime}:V(-s\omega_{1}-t\omega_{2})\otimes
V(a\omega_{1}+b\omega_{2})\rightarrow
V(-s\omega_{1}-t\omega_{2})\otimes V(a\omega_{1}+b\omega_{2})$ by
$\Theta^{\prime}(m\otimes m')=\sum\limits_{\nu}\Theta_{\nu}(m\otimes
m'),\forall~ m\in V(-s\omega_{1}-t\omega_{2}),m'\in
V(a\omega_{1}+b\omega_{2}).$ This is well-defined since only
finitely many terms of the sum are nonzero.

Now we define $\Psi:V(-s\omega_{1}-t\omega_{2})\otimes
V(a\omega_{1}+b\omega_{2})\rightarrow
V(-s\omega_{1}-t\omega_{2})\otimes V(a\omega_{1}+b\omega_{2})$ by
$\Psi(x)=\Theta^{\prime}(\bar{x}).$ In fact
$\Psi(\mathcal{L}_{\mathcal{A}})\subset\mathcal{L}_{\mathcal{A}}$
and $\Psi^{2}=1.$

From the definition, we have for all $b_1\in
\mathbf{B}(s\omega_{1}+t\omega_{2}),b_1'\in
\mathbf{B}(a\omega_{1}+b\omega_{2})$
$$\Psi(b_{1}^{+}\xi_{(-s,-t)}\otimes
b_1'^{-}\eta_{(a,b)})=\sum\limits_{\substack{ b_2\in
\mathbf{B}(s\omega_{1}+t\omega_{2})\\b_2'\in
\mathbf{B}(a\omega_{1}+b\omega_{2})}}\rho_{b_1,b_1';b_2,b_2'}b_{2}^{+}\xi_{(-s,-t)}\otimes
b_2'^{-}\eta_{(a,b)}.$$

where $\rho_{b_1,b_1';b_2,b_2'}\in \mathcal{A}$ and
$\rho_{b_1,b_1';b_2,b_2'}=0$ unless $(b_1,b_1')\geq (b_2,b_2');$
hence the last sum is finite.

Note also that $\rho_{b_1,b_{1}';b_1,b_{1}'}=1$ and
$$\sum\limits_{\substack{ b_2\in
\mathbf{B}(s\omega_{1}+t\omega_{2})\\b_2'\in
\mathbf{B}(a\omega_{1}+b\omega_{2})}}\bar{\rho}_{b_{1},b_{1}';b_{2},b_{2}'}
\rho_{b_2,b_{2}';b_3,b_{3}'}=\delta_{b_1,b_3}\delta_{b_{1}',b_{3}'}.$$

for any $b_1,b_3\in
\mathbf{B}(s\omega_{1}+t\omega_{2}),b_{1}',b_{3}'\in
\mathbf{B}(a\omega_{1}+b\omega_{2}).$

Applying [L5, Lemma 24.2.1], we see that there is a unique family of
elements $\pi_{b_{1},b_{1}';b_{2},b_{2}'} \in \mathbb{Z}[v^{-1}]$
defined for $b_1,b_2\in
\mathbf{B}(s\omega_{1}+t\omega_{2}),b_{1}',b_{2}'\in
\mathbf{B}(a\omega_{1}+b\omega_{2})$ such that
\begin{equation*}\pi_{b_{1},b_{1}';b_{1},b_{1}'}=1;\end{equation*}
\begin{equation*}\pi_{b_{1},b_{1}';b_{2},b_{2}'}\in v^{-1}\mathbb{Z}[v^{-1}], ~~\hbox{if }(b_1,b_1')\neq(b_2,b_2');\end{equation*}
\begin{equation*}\pi_{b_{1},b_{1}';b_{2},b_{2}'}=0, ~~\hbox{unless}~(b_1,b_1')\geq(b_2,b_2');\end{equation*}
\begin{equation*}\pi_{b_{1},b_{1}';b_{2},b_{2}'}=\sum\limits_{b_3,b_3'}\bar{\pi}_{b_{1},b_{1}';b_{3},b_{3}'}\rho_{b_3,b_{3}';b_2,b_{2}'},\
~~\hbox{for all }(b_1,b_1')\geq(b_2,b_2').\end{equation*}

\begin{theorem}{\rm (see [L5, Theorem~24.3.3])} $(a)$ For any $(b_1,b_1')\in
\mathbf{B}(s\omega_{1}+t\omega_{2})\times
\mathbf{B}(a\omega_{1}+b\omega_{2}) $, there is a unique element
$(b_1\lozenge b_1')_{(s,t),(a,b)}\in \mathcal{L}$ such
that$$\Psi((b_1\lozenge
b_1')_{(s,t),(a,b)})\hspace*{-0.1cm}=\hspace*{-0.1cm}(b_1\lozenge
b_1')_{(s,t),(a,b)},~~(b_1\lozenge
b_1')_{(s,t),(a,b)}-b_{1}^{+}\xi_{(-s,-t)}\otimes
b_1'^{-}\eta_{(a,b)}\hspace*{-0.1cm}\in\hspace*{-0.1cm}
v^{-1}\mathcal{L}.$$

$(b) $\begin{eqnarray*}
   && (b_1\lozenge b_1')_{(s,t),(a,b)} \\
   &=&\hspace*{-0.3cm}b_{1}^{+}\xi_{(-s,-t)}\hspace*{-0.1cm}\otimes\hspace*{-0.1cm}
b_1'^{-}\eta_{(a,b)}+\sum\limits_{\substack{ (b_2,b_2')\in
\mathbf{B}(s\omega_{1}+t\omega_{2})\times
\mathbf{B}(a\omega_{1}+b\omega_{2})\\(b_2,b_2')<(b_1,b_1')}}\hspace*{-0.2cm}\theta_{b_{1},b_{1}';b_{2},b_{2}'}b_{2}^{+}\xi_{(-s,-t)}\hspace*{-0.1cm}\otimes\hspace*{-0.1cm}
b_{2}'^{-}\eta_{(a,b).}
\end{eqnarray*}
where $\theta_{b_{1},b_{1}';b_{2},b_{2}'}\in
v^{-1}\mathbb{Z}[v^{-1}].$

$(c)$ These elements $(b_1\lozenge b_1')_{(s,t),(a,b)}$ with
$b_1,b_1'$ as above form a $\mathbb{Q}(v)$-basis of
$V(-s\omega_{1}-t\omega_{2})\otimes V(a\omega_{1}+b\omega_{2})$, a
$\mathbb{Z}[v^{-1}]$-basis of $\mathcal{L}$, and an
$\mathcal{A}$-basis of $\mathcal{L}_{\mathcal{A}}$.

$(d)$ The natural homomorphism
$\mathcal{L}\cap\Psi(\mathcal{L})\rightarrow
\mathcal{L}/v^{-1}\mathcal{L}$ is an isomorphism.\end{theorem}
\vskip1mm In fact the element $(b_1\lozenge
b_1')_{(s,t),(a,b)}=\sum\limits_{b_2,b_2'}\pi_{b_{1},b_{1}';b_{2},b_{2}'}b_{2}^{+}\xi_{(-s,-t)}\otimes
b_{2}'^{-}\eta_{(a,b)}$ satisfies the requirements of (a). The basis just defined as above is called the canonical basis of
$V(-s\omega_{1}-t\omega_{2})\otimes V(a\omega_{1}+b\omega_{2}).$

If we assume that $\zeta=(a,b)-(s,t),$ then
$u\mapsto u(\xi_{(-s,-t)}\otimes\eta_{(a,b)})$ defines a surjective
map $\dot{\mathbf{U}}1_{\zeta}\rightarrow
V(-s\omega_{1}-t\omega_{2})\otimes V(a\omega_{1}+b\omega_{2})$ (see [L5, Proposition 23.3.6]).

\begin{theorem}
{\rm (see [L5, Theorem~25.2.1])} Let $\zeta\in X$ and let $b,b''\in
\mathbf{B}.$

$(a)$ There is a unique element $u=b\lozenge_{\zeta} b''\in
\dot{\mathbf{U}}_{\mathcal{A}}1_{\zeta}$
 such that $$u(\xi_{(-s,-t)}\otimes\eta_{(a,b)})=(b\lozenge b'')_{(s,t),(a,b)}$$ for any $(s,t),(a,b)\in X^{+}$
 such that $b\in \mathbf{B}(s\omega_{1}+t\omega_{2}),b''\in
 \mathbf{B}(a\omega_{1}+b\omega_{2})$ and $(a,b)-(s,t)=\zeta.$

 $(b)$ If $(s,t),(a,b)\in X^{+}$ are such that $(a,b)-(s,t)=\zeta,$
 and either $b\notin  \mathbf{B}(s\omega_{1}+t\omega_{2})$ or $b''\notin
 \mathbf{B}(a\omega_{1}+b\omega_{2}),$ then $u(\xi_{(-s,-t)}\otimes\eta_{(a,b)})=0$ $($$u$ as in $(a)$$)$.

  $(c)$ The element $u$ in $(a)$ satisfies $\bar{u}=u.$

  $(d)$ These elements $b\lozenge_{\zeta} b'',$ for various $\zeta,b,b''$ as above form a $\mathbb{Q}(v)$-basis of
  $\dot{\mathbf{U}}$, and an $\mathcal{A}$-basis of
  $\dot{\mathbf{U}}_{\mathcal{A}}.$
\end{theorem}
\vskip1mm The basis of $\dot{\mathbf{U}}$ just defined is called the
canonical basis of $\dot{\mathbf{U}}$, and we denote it by
$\dot{\mathbf{B}}$.

In [L5], Lusztig had defined an important anti-automorphism $\sigma$
on a modified quantized enveloping algebra and conjectured that
elements of the canonical basis under this map also belong to the
canonical basis, which had been proved by Kashiwara in [K2, Theorem
4.3.2], so we have the following theorem in our case.

\begin{theorem} For any element $b$ $\in$ $\dot{\mathbf{B}}$, we
have $\sigma(b)$ $\in$ $\dot{\mathbf{B}}.$\end{theorem}

\section{A description of all elements of $\dot{\mathbf{B}}$}

The main result of this note is the following theorem which gives
all elements of the canonical basis $\dot{\mathbf{B}}$ of
$\dot{\mathbf{U}}.$
\medskip

\begin{theorem}  All elements of the canonical basis
$\dot{\mathbf{B}}$ are given by the following list:\\

$\mathbf{Part}$ $\mathbf{(1.1)}$\\

$e_{2}^{(h)}e_{1}^{(k)}e_{2}^{(j)}1_{(l,m)}f_{2}^{(u)}f_{1}^{(v)}f_{2}^{(w)},$
 \vskip2mm
 $\mbox{if}~-l \geq v+k-j-u,~~-m \geq u+j,~~k\geq h+j,~~v\geq u+w;\hfill(1)$\\

$\sum\limits_{0\leq p \leq j,u}(-1)^{p}\left[\begin{array}{c}m+u+j+p-1 \\
p
\end{array}\right]e_{2}^{(h)}e_{1}^{(k)}e_{2}^{(j-p)}1_{(l-p,m+2p)}f_{2}^{(u-p)}f_{1}^{(v)}f_{2}^{(w)},
$\vskip2mm $\mbox{if}~-l \geq v-u+k-j,~~u+j+(u+w-v)\leq -m \leq u+j,~~-m \geq
u+j+(j+h-k),~~k\geq h+j,~~v\geq u+w; \hfill(2)$\\

$\sum\limits_{\substack{0\leq p \leq j \\0\leq q \leq h\\0\leq p+q \leq u}}\hspace*{-0.2cm}(-1)^{p+q}\hspace*{-0.1cm}\left[\hspace*{-0.15cm}\begin{array}{c}m+u+j+q+p-1 \\
p
\end{array}\hspace*{-0.15cm}\right]\hspace*{-0.2cm}\left[\hspace*{-0.15cm}\begin{array}{c}m+u+j+(j+h-k)+q-1 \\
q
\end{array}\hspace*{-0.15cm}\right]$\vskip2mm$\times
e_{2}^{(h-q)}e_{1}^{(k)}e_{2}^{(j-p)}1_{(l-p-q,m+2p+2q)}f_{2}^{(u-p-q)}f_{1}^{(v)}f_{2}^{(w)},
$ \vskip2mm $\mbox{if}~-l \geq v-u+k-j,~~-m\geq u+j+(u+w-v),~~-m \leq
u+j+(j+h-k), ~~k\geq
h+j, ~~v\geq u+w;\hfill(3)$\\

$\sum\limits_{\substack{0\leq p\leq u\\0\leq q\leq w\\0\leq p+q \leq j}}(-1)^{p+q}\hspace*{-0.1cm}\left[\hspace*{-0.15cm}\begin{array}{c}u+j+m+q+p-1 \\
p
\end{array}\hspace*{-0.15cm}\right]\hspace*{-0.2cm}\left[\hspace*{-0.15cm}\begin{array}{c}u+j+m+u+w-v+q-1 \\
q
\end{array}\hspace*{-0.15cm}\right]$\vskip2mm$\times
e_{2}^{(h)}e_{1}^{(k)}e_{2}^{(j-p-q)}1_{(l-p-q,m+2p+2q)}f_{2}^{(u-p)}f_{1}^{(v)}f_{2}^{(w-q)},$
\vskip2mm $\mbox{if}~-l\geq v-u+k-j,~~-m \leq u+j+u+w-v,~~-m\geq
u+j+(j+h-k),~~k\geq
h+j,~~v\geq u+w;\hfill(4)$\\

$\sum\limits_{\substack{0\leq p\leq j\\ 0\leq q\leq w\\0\leq p+q\leq
j\\ 0\leq r\leq h\\0\leq p+r\leq u}}(-1)^{p+q+r}
\left[\hspace*{-0.2cm}\begin{array}{c}u\hspace*{-0.03cm}+\hspace*{-0.03cm}j\hspace*{-0.03cm}+\hspace*{-0.03cm}m\hspace*{-0.03cm}+\hspace*{-0.03cm}r\hspace*{-0.03cm}+\hspace*{-0.03cm}q\hspace*{-0.03cm}+\hspace*{-0.03cm}p\hspace*{-0.03cm}-\hspace*{-0.03cm}1 \\
p
\end{array}\hspace*{-0.2cm}\right]$\vskip2mm$\times \left[\hspace*{-0.2cm}\begin{array}{c}u\hspace*{-0.03cm}+\hspace*{-0.03cm}j\hspace*{-0.03cm}+\hspace*{-0.03cm}m\hspace*{-0.03cm}+\hspace*{-0.03cm}u\hspace*{-0.03cm}+\hspace*{-0.03cm}w\hspace*{-0.03cm}-\hspace*{-0.03cm}v\hspace*{-0.03cm}+\hspace*{-0.03cm}q\hspace*{-0.03cm}-\hspace*{-0.03cm}1 \\
q
\end{array}\hspace*{-0.2cm}\right]\hspace*{-0.1cm}\left[\hspace*{-0.2cm}\begin{array}{c}m\hspace*{-0.02cm}+\hspace*{-0.02cm}u\hspace*{-0.02cm}+\hspace*{-0.02cm}j\hspace*{-0.02cm}+\hspace*{-0.02cm}(j\hspace*{-0.02cm}+\hspace*{-0.02cm}h\hspace*{-0.02cm}-\hspace*{-0.02cm}k)\hspace*{-0.02cm}+\hspace*{-0.02cm}r\hspace*{-0.02cm}-\hspace*{-0.02cm}1 \\
r
\end{array}\hspace*{-0.2cm}\right]$\vskip2mm
$\times
e_{2}^{(h-r)}e_{1}^{(k)}e_{2}^{(j-p-q)}1_{(l-p-q-r,m+2p+2q+2r)}f_{2}^{(u-p-r)}f_{1}^{(v)}f_{2}^{(w-q)},$
\vskip2mm $\mbox{if}~-l\geq v-u+k-j,~~-m \leq
u+j+u+w-v,~~u+j+(j+h-k)+(u+w-v)\leq -m\leq
u+j+(j+h-k),~~k\geq h+j, ~~v\geq u+w;\hfill(5)$\\

$\sum\limits_{\substack{0\leq p,q,~ p+q \leq j\\0\leq r,i,~ r+i\leq
h\\0\leq p,r,~ p+r\leq u\\0\leq q,i, ~q+i\leq
w}}(-1)^{p+q+r+i}\hspace*{-0.1cm}
\left[\hspace*{-0.2cm}\begin{array}{c}u\hspace*{-0.06cm}+\hspace*{-0.06cm}j\hspace*{-0.06cm}+\hspace*{-0.06cm}m\hspace*{-0.06cm}+\hspace*{-0.06cm}r\hspace*{-0.06cm}+\hspace*{-0.06cm}2i\hspace*{-0.06cm}+\hspace*{-0.06cm}q\hspace*{-0.06cm}+\hspace*{-0.06cm}p\hspace*{-0.06cm}-\hspace*{-0.06cm}1 \\
p
\end{array}\hspace*{-0.2cm}\right]$\vskip2mm$\times\left[\hspace*{-0.2cm}\begin{array}{c}u\hspace*{-0.06cm}+\hspace*{-0.06cm}j\hspace*{-0.06cm}+\hspace*{-0.06cm}m\hspace*{-0.06cm}+\hspace*{-0.06cm}u\hspace*{-0.06cm}+\hspace*{-0.06cm}w\hspace*{-0.06cm}-\hspace*{-0.06cm}v\hspace*{-0.06cm}+\hspace*{-0.06cm}i\hspace*{-0.06cm}+\hspace*{-0.06cm}q\hspace*{-0.06cm}-\hspace*{-0.06cm}1 \\
q
\end{array}\hspace*{-0.2cm}\right]\hspace*{-0.1cm}\left[\hspace*{-0.2cm}\begin{array}{c}m\hspace*{-0.06cm}+\hspace*{-0.06cm}u\hspace*{-0.06cm}+\hspace*{-0.06cm}j\hspace*{-0.06cm}+\hspace*{-0.06cm}j\hspace*{-0.06cm}+\hspace*{-0.06cm}h\hspace*{-0.06cm}-\hspace*{-0.06cm}k\hspace*{-0.06cm}+\hspace*{-0.06cm}i\hspace*{-0.06cm}+\hspace*{-0.06cm}r\hspace*{-0.06cm}-\hspace*{-0.06cm}1 \\
r
\end{array}\hspace*{-0.2cm}\right]$\vskip2mm$\times\left[\hspace*{-0.2cm}\begin{array}{c}m\hspace*{-0.06cm}+\hspace*{-0.06cm}u\hspace*{-0.06cm}+\hspace*{-0.06cm}j+\hspace*{-0.06cm}j\hspace*{-0.06cm}+\hspace*{-0.06cm}h\hspace*{-0.08cm}-\hspace*{-0.06cm}k\hspace*{-0.06cm}+\hspace*{-0.06cm}u\hspace*{-0.06cm}+\hspace*{-0.06cm}w\hspace*{-0.06cm}-\hspace*{-0.06cm}v\hspace*{-0.06cm}+\hspace*{-0.06cm}i\hspace*{-0.06cm}-\hspace*{-0.06cm}1 \\
i
\end{array}\hspace*{-0.2cm}\right]$\vskip2mm$\times
e_{2}^{(h-r-i)}e_{1}^{(k)}e_{2}^{(j-p-q)}1_{(l-p-q-r-i,m+2p+2q+2r+2i)}f_{2}^{(u-p-r)}f_{1}^{(v)}f_{2}^{(w-q-i)},
$ \vskip2mm  $\mbox{if}~-m\leq u+j+(j+h-k)+(u+w-v),~~-l-m\geq
j+h+u+w,~~k\geq h+j,
~~v\geq u+w;\hfill(6)$\\

$\sum\limits_{\substack{0\leq z\leq k,v\\0\leq p,q,~ p+q \leq
j\\0\leq r,i,~ r+i+z\leq h\\0\leq p,r,~ p+r\leq u\\0\leq q,i,
~q+i+z\leq w}}(-1)^{p+q+r+i+z}\hspace*{-0.1cm}
\left[\hspace*{-0.2cm}\begin{array}{c}u\hspace*{-0.06cm}+\hspace*{-0.06cm}j\hspace*{-0.06cm}+\hspace*{-0.06cm}m\hspace*{-0.06cm}+\hspace*{-0.06cm}r\hspace*{-0.06cm}+\hspace*{-0.06cm}2i\hspace*{-0.06cm}+\hspace*{-0.06cm}q\hspace*{-0.06cm}+\hspace*{-0.06cm}z\hspace*{-0.06cm}+\hspace*{-0.06cm}p\hspace*{-0.06cm}-\hspace*{-0.06cm}1 \\
p
\end{array}\hspace*{-0.2cm}\right]$\vskip2mm$ \hspace*{-0.6cm}\times\left[\hspace*{-0.2cm}\begin{array}{c}u\hspace*{-0.06cm}+\hspace*{-0.06cm}j\hspace*{-0.06cm}+\hspace*{-0.06cm}m\hspace*{-0.06cm}+\hspace*{-0.06cm}u\hspace*{-0.06cm}+\hspace*{-0.06cm}w\hspace*{-0.06cm}-\hspace*{-0.06cm}v\hspace*{-0.06cm}+\hspace*{-0.06cm}z\hspace*{-0.06cm}+\hspace*{-0.06cm}i\hspace*{-0.06cm}+\hspace*{-0.06cm}q\hspace*{-0.06cm}-\hspace*{-0.06cm}1 \\
q
\end{array}\hspace*{-0.2cm}\right]\hspace*{-0.1cm}\left[\hspace*{-0.2cm}\begin{array}{c}m\hspace*{-0.06cm}+\hspace*{-0.06cm}u\hspace*{-0.06cm}+\hspace*{-0.06cm}j\hspace*{-0.06cm}+\hspace*{-0.06cm}(j\hspace*{-0.06cm}+\hspace*{-0.06cm}h\hspace*{-0.06cm}-\hspace*{-0.06cm}k)\hspace*{-0.06cm}+\hspace*{-0.06cm}z\hspace*{-0.06cm}+\hspace*{-0.06cm}i\hspace*{-0.06cm}+\hspace*{-0.06cm}r\hspace*{-0.06cm}-\hspace*{-0.06cm}1 \\
r
\end{array}\hspace*{-0.2cm}\right]$\vskip2mm$\hspace*{-0.6cm}\times\left[\hspace*{-0.2cm}\begin{array}{c}m\hspace*{-0.06cm}+\hspace*{-0.06cm}u\hspace*{-0.06cm}+\hspace*{-0.06cm}j+\hspace*{-0.06cm}(j\hspace*{-0.06cm}+\hspace*{-0.06cm}h\hspace*{-0.08cm}-\hspace*{-0.06cm}k)\hspace*{-0.06cm}+\hspace*{-0.06cm}(u\hspace*{-0.06cm}+\hspace*{-0.06cm}w\hspace*{-0.06cm}-\hspace*{-0.06cm}v)\hspace*{-0.06cm}+\hspace*{-0.06cm}z\hspace*{-0.06cm}+\hspace*{-0.06cm}i\hspace*{-0.06cm}-\hspace*{-0.06cm}1 \\
i
\end{array}\hspace*{-0.2cm}\right]\hspace*{-0.1cm}\left[\hspace*{-0.2cm}\begin{array}{c}l\hspace*{-0.06cm}+\hspace*{-0.06cm}m\hspace*{-0.06cm}+\hspace*{-0.06cm}j\hspace*{-0.06cm}+\hspace*{-0.06cm}h\hspace*{-0.06cm}+\hspace*{-0.06cm}u\hspace*{-0.06cm}+\hspace*{-0.06cm}w\hspace*{-0.06cm}+\hspace*{-0.06cm}z\hspace*{-0.06cm}-\hspace*{-0.06cm}1 \\
z
\end{array}\hspace*{-0.2cm}\right]$\vskip2mm
$\hspace*{-0.6cm}\times
e_{2}^{(h-r-i-z)}e_{1}^{(k-z)}e_{2}^{(j-p-q)}1_{(l-p-q-r-i+z,m+2p+2q+2r+2i+z)}
f_{2}^{(u-p-r)}f_{1}^{(v-z)}f_{2}^{(w-q-i-z)},$ \vskip2mm
$\mbox{if}~-l\geq v-u+k-j,~~-l-m\leq j+h+u+w,~~k\geq h+j, ~~v\geq u+w;\hfill(7)$\\

$e_{2}^{(h)}e_{1}^{(k)}e_{2}^{(j)}1_{(l,m)}f_{1}^{(u)}f_{2}^{(v)}f_{1}^{(w)},$\vskip2mm$
\mbox{if}~-l \geq u+k-j,~~-m \geq j+v-u,~~k\geq h+j,~~v\geq u+w;\hfill(8)$\\

$\sum\limits_{\substack{0\leq p \leq j,v}}(-1)^{p}
\hspace*{-0.1cm}\left[\hspace*{-0.1cm}\begin{array}{c}m+j+v-u+p-1 \\
p
\end{array}\hspace*{-0.1cm}\right]$
$\hspace*{-0.1cm}e_{2}^{(h)}e_{1}^{(k)}e_{2}^{(j-p)}1_{(l-p,m+2p)}f_{1}^{(u)}f_{2}^{(v-p)}f_{1}^{(w)},$
\vskip2mm $\mbox{if}~-l\geq u+k-j,~~v+j-u+(j+h-k)\leq -m\leq
v+j-u,~~k\geq h+j, ~~v\geq
u+w;\hfill(9)$\\

$\sum\limits_{\substack{0\leq p \leq k,u}}(-1)^{p}
\left[\hspace*{-0.1cm}\begin{array}{c}l+u+k-j+p-1 \\
p
\end{array}\hspace*{-0.1cm}\right]$
$\hspace*{-0.1cm}e_{2}^{(h)}e_{1}^{(k-p)}e_{2}^{(j)}1_{(l+2p,m-p)}f_{1}^{(u-p)}f_{2}^{(v)}f_{1}^{(w)},$
\vskip2mm $\mbox{if}~u+k-j+(u+w-v)\leq -l\leq u+k-j,~~-m\geq
v+j-u,~~k\geq h+j,~~ v\geq
u+w;\hfill(10)$\\

$\sum\limits_{\substack{0\leq p \leq j,v\\0\leq q\leq
k,u}}(-1)^{p+q}
\hspace*{-0.1cm}\left[\begin{array}{c}m+j+v-u+p-1 \\
p
\end{array}\right]\hspace*{-0.14cm}\left[\begin{array}{c}l+u+k-j+q-1 \\
q
\end{array}\right]$\vskip2mm
$\times
e_{2}^{(h)}e_{1}^{(k-q)}e_{2}^{(j-p)}1_{(l-p+2q,m+2p-q)}f_{1}^{(u-q)}f_{2}^{(v-p)}f_{1}^{(w)},$
\vskip2mm $\mbox{if}~u+k-j+(u+w-v)\leq -l\leq
u+k-j,~~v+j-u+(j+h-k)\leq -m\leq v+j-u,~~k\geq
h+j, ~~v\geq u+w;\hfill(11)$\\

$\sum\limits_{\substack{0\leq p\leq u\\0\leq q\leq w\\0\leq p+q \leq k}}(-1)^{p+q}\hspace*{-0.1cm}\left[\hspace*{-0.2cm}\begin{array}{c}u\hspace*{-0.05cm}+\hspace*{-0.05cm}l\hspace*{-0.05cm}+\hspace*{-0.05cm}k\hspace*{-0.05cm}-\hspace*{-0.05cm}j\hspace*{-0.05cm}+\hspace*{-0.05cm}q\hspace*{-0.05cm}+\hspace*{-0.05cm}p\hspace*{-0.05cm}-\hspace*{-0.05cm}1 \\
p
\end{array}\hspace*{-0.2cm}\right]\hspace*{-0.1cm}\left[\hspace*{-0.2cm}\begin{array}{c}u\hspace*{-0.05cm}+\hspace*{-0.05cm}l\hspace*{-0.05cm}+\hspace*{-0.05cm}k\hspace*{-0.05cm}-\hspace*{-0.05cm}j\hspace*{-0.05cm}+\hspace*{-0.05cm}u\hspace*{-0.05cm}+\hspace*{-0.05cm}w\hspace*{-0.05cm}-\hspace*{-0.05cm}v\hspace*{-0.05cm}+\hspace*{-0.05cm}q\hspace*{-0.05cm}-\hspace*{-0.05cm}1 \\
q
\end{array}\hspace*{-0.2cm}\right]$\vskip2mm$\times e_{2}^{(h)}e_{1}^{(k-p-q)}e_{2}^{(j)}1_{(l+2p+2q,m-p-q)}
f_{1}^{(u-p)}f_{2}^{(v)}f_{1}^{(w-q)},$ \vskip2mm
 $ \mbox{if}~-l-m \geq u+w+k, ~~-l \leq
u+k-j+(u+w-v),~~k\geq h+j,~~v\geq u+w;\hfill(12)$\\

$\sum\limits_{\substack{0\leq r\leq j,v\\0\leq p\leq u\\0\leq q,~q+r\leq w\\0\leq p+q+r \leq k}}(-1)^{p+q+r}\hspace*{-0.1cm}\left[\hspace*{-0.2cm}\begin{array}{c}u\hspace*{-0.05cm}+\hspace*{-0.05cm}l\hspace*{-0.05cm}+\hspace*{-0.05cm}k\hspace*{-0.05cm}-\hspace*{-0.05cm}j\hspace*{-0.05cm}+\hspace*{-0.05cm}r\hspace*{-0.05cm}+\hspace*{-0.05cm}q\hspace*{-0.05cm}+\hspace*{-0.05cm}p\hspace*{-0.05cm}-\hspace*{-0.05cm}1 \\
p
\end{array}\hspace*{-0.2cm}\right]$\vskip2mm$\times\left[\hspace*{-0.2cm}\begin{array}{c}u\hspace*{-0.05cm}+\hspace*{-0.05cm}l\hspace*{-0.05cm}+\hspace*{-0.05cm}k\hspace*{-0.05cm}-\hspace*{-0.05cm}j\hspace*{-0.05cm}+\hspace*{-0.05cm}u\hspace*{-0.05cm}+\hspace*{-0.05cm}w\hspace*{-0.05cm}-\hspace*{-0.05cm}v\hspace*{-0.05cm}+\hspace*{-0.05cm}r\hspace*{-0.05cm}+\hspace*{-0.05cm}q\hspace*{-0.05cm}-\hspace*{-0.05cm}1 \\
q
\end{array}\hspace*{-0.2cm}\right]\hspace*{-0.1cm}\left[\hspace*{-0.2cm}\begin{array}{c}u+w+l+m+k+r-1 \\
r
\end{array}\hspace*{-0.2cm}\right]$\vskip2mm$\times e_{2}^{(h)}e_{1}^{(k-p-q-r)}e_{2}^{(j-r)}1_{(l+2p+2q+r,m-p-q+r)}
f_{1}^{(u-p)}f_{2}^{(v-r)}f_{1}^{(w-q-r)},$\vskip2mm
$ \mbox{if}~-l-m \leq u+w+k,~~ -m \geq v-u+j,~~ k\geq h+j,~~v\geq u+w;\hfill(13)$\\

$\mathbf{Part}$ $\mathbf{(1.2)}$\\

$f_{2}^{(u)}f_{1}^{(v)}f_{2}^{(w)}1_{(l,m)}e_{2}^{(h)}e_{1}^{(k)}e_{2}^{(j)},$\vskip2mm$
\mbox{if}~-l \leq w-v+h-k,~~-m \leq -w-h,~~k\geq h+j,~~v\geq u+w;\hfill(1')$\\

$\sum\limits_{0\leq p \leq h,w}(-1)^{p}\left[\begin{array}{c}w-m+h+p-1 \\
p
\end{array}\right]f_{2}^{(u)}f_{1}^{(v)}f_{2}^{(w-p)}1_{(l+p,m-2p)}
e_{2}^{(h-p)}e_{1}^{(k)}e_{2}^{(j)},$\vskip2mm  $\mbox{if}~-l \leq
w-v+h-k,~~-w-h\leq -m
\leq -w-h+v-u-w,~~-m\leq -w-h+(k-j-h), ~~k\geq h+j,~~v\geq u+w;\hfill(2')$\\

$\sum\limits_{\substack{0\leq p \leq h\\0\leq q\leq j\\0\leq p+q\leq w}}(-1)^{p+q}\hspace*{-0.1cm}\left[\hspace*{-0.15cm}\begin{array}{c}w-m+h+q+p-1 \\
p
\end{array}\hspace*{-0.15cm}\right]\hspace*{-0.2cm}\left[\hspace*{-0.15cm}\begin{array}{c}w-m+h+j+h-k+q-1 \\
p
\end{array}\hspace*{-0.15cm}\right]$\vskip2mm$\times f_{2}^{(u)}f_{1}^{(v)}f_{2}^{(w-p-q)}1_{(l+p+q,m-2p-2q)}
e_{2}^{(h-p)}e_{1}^{(k)}e_{2}^{(j-q)},$ \vskip2mm $\mbox{if}~-l \leq
w-v+h-k,~~ -m \leq
-w-h+v-u-w,~~-m\geq -w-h+(k-j-h), ~~k\geq h+j,~~v\geq u+w;\hfill(3')$\\

$\sum\limits_{\substack{0\leq p \leq w\\0\leq q\leq u\\0\leq p+q \leq h}}(-1)^{p+q}\hspace*{-0.1cm}\left[\hspace*{-0.2cm}\begin{array}{c}w-m+h+q+p-1 \\
p
\end{array}\hspace*{-0.2cm}\right]\hspace*{-0.2cm}\left[\hspace*{-0.2cm}\begin{array}{c}w-m+h+u+w-v+q-1 \\
q
\end{array}\hspace*{-0.2cm}\right]$\vskip2mm$\times
f_{2}^{(u-q)}f_{1}^{(v)}f_{2}^{(w-p)}
1_{(l+p+q,m-2p-2q)}e_{2}^{(h-p-q)}e_{1}^{(k)}e_{2}^{(j)},$ \vskip2mm
$\mbox{if}~-l\leq w-v+h-k,~~-m \geq -w-h+v-u-w,~~-m\leq
-w-h+(k-j-h),~~k\geq h+j,~~
v\geq u+w;\hfill(4')$\\

$\sum\limits_{\substack{0\leq p\leq w\\0\leq q\leq u\\0\leq p+q \leq h\\0\leq r \leq j\\0\leq p+r\leq w}}
(-1)^{p+q+r}\left[\hspace*{-0.2cm}\begin{array}{c}w\hspace*{-0.03cm}-\hspace*{-0.03cm}m\hspace*{-0.03cm}+\hspace*{-0.03cm}h\hspace*{-0.03cm}+\hspace*{-0.03cm}r\hspace*{-0.03cm}+\hspace*{-0.03cm}q\hspace*{-0.03cm}+\hspace*{-0.03cm}p\hspace*{-0.03cm}-\hspace*{-0.03cm}1 \\
p
\end{array}\hspace*{-0.2cm}\right]$\vskip2mm$\times\left[\hspace*{-0.2cm}\begin{array}{c}w\hspace*{-0.03cm}-\hspace*{-0.03cm}m\hspace*{-0.03cm}+\hspace*{-0.03cm}h\hspace*{-0.03cm}+\hspace*{-0.03cm}u\hspace*{-0.03cm}+\hspace*{-0.03cm}w\hspace*{-0.03cm}-\hspace*{-0.03cm}v\hspace*{-0.03cm}+\hspace*{-0.03cm}q\hspace*{-0.03cm}-\hspace*{-0.03cm}1 \\
q
\end{array}\hspace*{-0.2cm}\right]\hspace*{-0.1cm}\left[\hspace*{-0.2cm}\begin{array}{c}w\hspace*{-0.03cm}-\hspace*{-0.03cm}m\hspace*{-0.03cm}+\hspace*{-0.03cm}h\hspace*{-0.03cm}+\hspace*{-0.03cm}(j\hspace*{-0.03cm}+\hspace*{-0.03cm}h\hspace*{-0.03cm}-\hspace*{-0.03cm}k)\hspace*{-0.03cm}+\hspace*{-0.03cm}r\hspace*{-0.03cm}-\hspace*{-0.03cm}1 \\
r
\end{array}\hspace*{-0.2cm}\right]$\vskip2mm
$\times f_{2}^{(u-q)}f_{1}^{(v)}f_{2}^{(w-p-r)}
1_{(l+p+q+r,m-2p-2q-2r)}e_{2}^{(h-p-q)}e_{1}^{(k)}e_{2}^{(j-r)},$
\vskip2mm $\mbox{if}~-l\leq w-v+h-k,~~-m \geq
-w-h+v-u-w,~~-w-h+(k-j-h)\leq -m \leq -w-h+(k-j-h)+(v-u-w),~~k\geq
h+j,~~
v\geq u+w;\hfill(5')$\\

$\sum\limits_{\substack{0\leq p,q,~p+q \leq h\\0\leq r,i,~r+i \leq
j\\0\leq q,i,~q+i\leq u\\0\leq p,r,~p+r\leq
w}}(-1)^{p+q+r+i}\left[\hspace*{-0.2cm}\begin{array}{c}w\hspace*{-0.06cm}+\hspace*{-0.06cm}h\hspace*{-0.06cm}-\hspace*{-0.06cm}m\hspace*{-0.06cm}+\hspace*{-0.06cm}r\hspace*{-0.06cm}+\hspace*{-0.06cm}2i\hspace*{-0.06cm}+\hspace*{-0.06cm}q\hspace*{-0.06cm}+\hspace*{-0.06cm}p\hspace*{-0.06cm}-\hspace*{-0.06cm}1 \\
p
\end{array}\hspace*{-0.2cm}\right]$\vskip2mm$\times\left[\hspace*{-0.2cm}\begin{array}{c}w\hspace*{-0.06cm}+\hspace*{-0.06cm}h\hspace*{-0.06cm}-\hspace*{-0.06cm}m\hspace*{-0.06cm}+\hspace*{-0.06cm}u\hspace*{-0.06cm}+\hspace*{-0.06cm}w\hspace*{-0.06cm}-\hspace*{-0.06cm}v\hspace*{-0.06cm}+\hspace*{-0.06cm}i\hspace*{-0.06cm}+\hspace*{-0.06cm}q\hspace*{-0.06cm}-\hspace*{-0.06cm}1 \\
q
\end{array}\hspace*{-0.2cm}\right]\hspace*{-0.1cm}\left[\hspace*{-0.2cm}\begin{array}{c}w\hspace*{-0.06cm}-\hspace*{-0.06cm}m\hspace*{-0.06cm}+\hspace*{-0.06cm}h\hspace*{-0.06cm}+\hspace*{-0.06cm}(j\hspace*{-0.06cm}+\hspace*{-0.06cm}h\hspace*{-0.06cm}-\hspace*{-0.06cm}k)\hspace*{-0.06cm}+\hspace*{-0.06cm}i\hspace*{-0.06cm}+\hspace*{-0.06cm}r\hspace*{-0.06cm}-\hspace*{-0.06cm}1 \\
r
\end{array}\hspace*{-0.2cm}\right]$\vskip2mm$\times\left[\hspace*{-0.2cm}\begin{array}{c}w\hspace*{-0.06cm}-\hspace*{-0.06cm}m\hspace*{-0.06cm}+\hspace*{-0.06cm}h\hspace*{-0.06cm}+\hspace*{-0.06cm}(j\hspace*{-0.06cm}+\hspace*{-0.06cm}h\hspace*{-0.06cm}-\hspace*{-0.06cm}k)\hspace*{-0.06cm}+\hspace*{-0.06cm}(u\hspace*{-0.06cm}+\hspace*{-0.06cm}w\hspace*{-0.06cm}-\hspace*{-0.06cm}v)\hspace*{-0.06cm}+\hspace*{-0.06cm}i\hspace*{-0.06cm}-\hspace*{-0.06cm}1 \\
i
\end{array}\hspace*{-0.2cm}\right]$\vskip2mm$\times f_{2}^{(u-q-i)}f_{1}^{(v)}f_{2}^{(w-p-r)}
1_{(l+p+q+r+i,m-2p-2q-2r-2i)}e_{2}^{(h-p-q)}e_{1}^{(k)}e_{2}^{(j-r-i)},$
\vskip2mm $\mbox{if}~-m\geq -w-h+(k-j-h)+(v-u-w),~~-l-m\leq
-j-h-u-w,~~k\geq h+j,~~
v\geq u+w;\hfill(6')$\\

$\sum\limits_{\substack{0\leq z\leq k,v\\0\leq p,q,~p+q \leq
h\\0\leq r,i,~r+i+z \leq j\\0\leq q,i,~q+i+z\leq u\\0\leq
p,r,~p+r\leq
w}}(-1)^{p+q+r+i+z}\left[\hspace*{-0.2cm}\begin{array}{c}w\hspace*{-0.06cm}+\hspace*{-0.06cm}h\hspace*{-0.06cm}-\hspace*{-0.06cm}m\hspace*{-0.06cm}+\hspace*{-0.06cm}r\hspace*{-0.06cm}+\hspace*{-0.06cm}2i\hspace*{-0.06cm}+\hspace*{-0.06cm}q\hspace*{-0.06cm}+\hspace*{-0.06cm}z\hspace*{-0.06cm}+\hspace*{-0.06cm}p\hspace*{-0.06cm}-\hspace*{-0.06cm}1 \\
p
\end{array}\hspace*{-0.2cm}\right]$\vskip2mm$\hspace*{-0.6cm}\times\left[\hspace*{-0.2cm}\begin{array}{c}w\hspace*{-0.06cm}+\hspace*{-0.06cm}h\hspace*{-0.06cm}-\hspace*{-0.06cm}m\hspace*{-0.06cm}+\hspace*{-0.06cm}u\hspace*{-0.06cm}+\hspace*{-0.06cm}w\hspace*{-0.06cm}-\hspace*{-0.06cm}v\hspace*{-0.06cm}+\hspace*{-0.06cm}z\hspace*{-0.06cm}+\hspace*{-0.06cm}i\hspace*{-0.06cm}+\hspace*{-0.06cm}q\hspace*{-0.06cm}-\hspace*{-0.06cm}1 \\
q
\end{array}\hspace*{-0.2cm}\right]\hspace*{-0.1cm}\left[\hspace*{-0.2cm}\begin{array}{c}w\hspace*{-0.06cm}-\hspace*{-0.06cm}m\hspace*{-0.06cm}+\hspace*{-0.06cm}h\hspace*{-0.06cm}+\hspace*{-0.06cm}(j\hspace*{-0.06cm}+\hspace*{-0.06cm}h\hspace*{-0.06cm}-\hspace*{-0.06cm}k)\hspace*{-0.06cm}+\hspace*{-0.06cm}z\hspace*{-0.06cm}+\hspace*{-0.06cm}i\hspace*{-0.06cm}+\hspace*{-0.06cm}r\hspace*{-0.06cm}-\hspace*{-0.06cm}1 \\
r
\end{array}\hspace*{-0.2cm}\right]$\vskip2mm$\hspace*{-0.6cm}\times\left[\hspace*{-0.2cm}\begin{array}{c}w\hspace*{-0.06cm}-\hspace*{-0.06cm}m\hspace*{-0.06cm}+\hspace*{-0.06cm}h\hspace*{-0.06cm}+\hspace*{-0.06cm}j\hspace*{-0.06cm}+\hspace*{-0.06cm}h\hspace*{-0.06cm}-\hspace*{-0.06cm}k\hspace*{-0.06cm}+\hspace*{-0.06cm}u\hspace*{-0.06cm}+\hspace*{-0.06cm}w\hspace*{-0.06cm}-\hspace*{-0.06cm}v\hspace*{-0.06cm}+\hspace*{-0.06cm}z\hspace*{-0.06cm}+\hspace*{-0.06cm}i\hspace*{-0.06cm}-\hspace*{-0.06cm}1 \\
i
\end{array}\hspace*{-0.2cm}\right]\hspace*{-0.1cm}\left[\hspace*{-0.2cm}\begin{array}{c}-l\hspace*{-0.06cm}-\hspace*{-0.06cm}m\hspace*{-0.06cm}+\hspace*{-0.06cm}j\hspace*{-0.06cm}+\hspace*{-0.06cm}h\hspace*{-0.06cm}+\hspace*{-0.06cm}u\hspace*{-0.06cm}+\hspace*{-0.06cm}w\hspace*{-0.06cm}+\hspace*{-0.06cm}z\hspace*{-0.06cm}-\hspace*{-0.06cm}1 \\
z
\end{array}\hspace*{-0.2cm}\right]$\vskip2mm$\hspace*{-0.6cm}\times f_{2}^{(u-q-i-z)}f_{1}^{(v-z)}f_{2}^{(w-p-r)}
1_{(l+p+q+r+i-z,m-2p-2q-2r-2i-z)}e_{2}^{(h-p-q)}e_{1}^{(k-z)}e_{2}^{(j-r-i-z)},$
\vskip2mm $\mbox{if}~-l\leq w-v+h-k,~~-l-m\geq -j-h-u-w,~~k\geq h+j, ~~v\geq u+w;\hfill(7')$\\

$f_{1}^{(u)}f_{2}^{(v)}f_{1}^{(w)}1_{(l,m)}e_{2}^{(h)}e_{1}^{(k)}e_{2}^{(j)},$\vskip2mm$
\mbox{if}~-l \leq h-k-w,~~-m \leq w-v-h,~~k\geq h+j,~~v\geq u+w;\hfill(8')$\\

$\sum\limits_{\substack{0\leq p \leq h,v}}\hspace*{-0.1cm}(-1)^{p}\hspace*{-0.1cm}\left[\hspace*{-0.17cm}\begin{array}{c}-m+h+v-w+p-1 \\
p
\end{array}\hspace*{-0.17cm}\right]$
$\hspace*{-0.13cm}f_{1}^{(u)}f_{2}^{(v-p)}f_{1}^{(w)}
1_{(l+p,m-2p)}e_{2}^{(h-p)}e_{1}^{(k)}e_{2}^{(j)},$ \vskip2mm
$\mbox{if}~-l\leq -w+h-k, ~~-h+w-v\leq -m \leq
-h+w-v+(k-j-h),~~k\geq h+j,~~
v\geq u+w;\hfill(9')$\\

$\sum\limits_{\substack{0\leq p \leq k,w}}\hspace*{-0.1cm}(-1)^{p}\hspace*{-0.1cm}\left[\hspace*{-0.14cm}\begin{array}{c}-l+w+k-h+p-1 \\
p
\end{array}\hspace*{-0.14cm}\right]$
$\hspace*{-0.13cm}f_{1}^{(u)}f_{2}^{(v)}f_{1}^{(w-p)}
1_{(l-2p,m+p)}e_{2}^{(h)}e_{1}^{(k-p)}e_{2}^{(j)},$ \vskip2mm
$\mbox{if}~-w+h-k\leq -l\leq -w+h-k+(v-u-w), ~~-m \leq
-h+w-v,~~k\geq h+j,~~
v\geq u+w;\hfill(10')$\\

$\sum\limits_{\substack{0\leq p \leq h,v\\0\leq q\leq
k,w}}(-1)^{p+q}\left[\hspace*{-0.17cm}\begin{array}{c}-m+h+v-w+p-1 \\
p
\end{array}\hspace*{-0.17cm}\right]\hspace*{-0.1cm}\left[\hspace*{-0.14cm}\begin{array}{c}-l+w+k-h+p-1 \\
p
\end{array}\hspace*{-0.14cm}\right]$\vskip2mm$\times f_{1}^{(u)}f_{2}^{(v-p)}f_{1}^{(w-q)}
1_{(l+p-2q,m-2p+q)}e_{2}^{(h-p)}e_{1}^{(k-q)}e_{2}^{(j)},$ \vskip2mm
$\mbox{if}~-w+h-k\leq -l\leq -w+h-k+(v-u-w),~~ -h+w-v\leq -m \leq
-h+w-v+(k-j-h),~~k\geq h+j, ~~v\geq u+w;\hfill(11')$\\

$\sum\limits_{\substack{0\leq p \leq w\\0\leq q\leq u\\0\leq p+q \leq k}}(-1)^{p+q}\hspace*{-0.1cm}\left[\hspace*{-0.2cm}\begin{array}{c}w\hspace*{-0.05cm}-\hspace*{-0.05cm}l\hspace*{-0.05cm}+\hspace*{-0.05cm}k\hspace*{-0.05cm}-\hspace*{-0.05cm}h\hspace*{-0.05cm}+\hspace*{-0.05cm}q\hspace*{-0.05cm}+\hspace*{-0.05cm}p\hspace*{-0.05cm}-\hspace*{-0.05cm}1 \\
p
\end{array}\hspace*{-0.2cm}\right]\hspace*{-0.1cm}\left[\hspace*{-0.2cm}\begin{array}{c}w\hspace*{-0.05cm}-\hspace*{-0.05cm}l\hspace*{-0.05cm}+\hspace*{-0.05cm}k\hspace*{-0.05cm}-\hspace*{-0.05cm}h\hspace*{-0.05cm}+\hspace*{-0.05cm}u+\hspace*{-0.05cm}w\hspace*{-0.05cm}-\hspace*{-0.05cm}v\hspace*{-0.05cm}+\hspace*{-0.05cm}q-\hspace*{-0.05cm}1 \\
q
\end{array}\hspace*{-0.2cm}\right]$\vskip2mm$\times f_{1}^{(u-q)}f_{2}^{(v)}f_{1}^{(w-p)}1_{(l-2p-2q,m+p+q)}
e_{2}^{(h)}e_{1}^{(k-p-q)}e_{2}^{(j)},$ \vskip2mm $\mbox{if}~ -l-m
\leq
-u-w-k, ~~-l \geq -w+h-k+(v-u-w), ~~k\geq  h+j, ~~v\geq u+w;\hfill(12')$\\

$\sum\limits_{\substack{0\leq r\leq h,v\\0\leq p \leq w\\0\leq q,~q+r\leq u\\0\leq p+q+r \leq k}}(-1)^{p+q+r}\hspace*{-0.1cm}\left[\hspace*{-0.2cm}\begin{array}{c}w\hspace*{-0.05cm}-\hspace*{-0.05cm}l\hspace*{-0.05cm}+\hspace*{-0.05cm}k\hspace*{-0.05cm}-\hspace*{-0.05cm}h\hspace*{-0.05cm}+\hspace*{-0.05cm}r\hspace*{-0.05cm}+\hspace*{-0.05cm}q\hspace*{-0.05cm}+\hspace*{-0.05cm}p\hspace*{-0.05cm}-\hspace*{-0.05cm}1 \\
p
\end{array}\hspace*{-0.2cm}\right]$\vskip2mm$\times\left[\hspace*{-0.2cm}\begin{array}{c}w\hspace*{-0.05cm}-\hspace*{-0.05cm}l\hspace*{-0.05cm}+\hspace*{-0.05cm}k\hspace*{-0.05cm}-\hspace*{-0.05cm}h\hspace*{-0.05cm}+\hspace*{-0.05cm}u+\hspace*{-0.05cm}w\hspace*{-0.05cm}-\hspace*{-0.05cm}v\hspace*{-0.05cm}+\hspace*{-0.05cm}q-\hspace*{-0.05cm}1 \\
q
\end{array}\hspace*{-0.2cm}\right]\hspace*{-0.1cm}\left[\hspace*{-0.2cm}\begin{array}{c}u+w-l-m+k+r-1 \\
r
\end{array}\hspace*{-0.2cm}\right]$\vskip2mm$\times f_{1}^{(u-q-r)}f_{2}^{(v-r)}f_{1}^{(w-p)}1_{(l-2p-2q-r,m+p+q-r)}
e_{2}^{(h-r)}e_{1}^{(k-p-q-r)}e_{2}^{(j)},$ \vskip2mm $\mbox{if}~
-l-m \geq
-u-w-k,~~ -m\leq w-v-h,~~ k\geq  h+j, ~~v\geq u+w.\hfill(13')$\\

$\mathbf{Part}$ $\mathbf{(2)}$ Using the symmetries of the indices {\rm 1} and {\rm 2}, we can also write the other {\rm 26} elements and we will omit them here.
\end{theorem}

\section{The proof of Theorem 3.1} This section is devoted to the proof of Theorem 3.1. We need some
combinatorial identities, which will be used in the proof of Theorem
3.1.

\begin{lemma}
$(a)$ Assume that $n,r\in \mathbb{N},~~m\in \mathbb{Z}.$ We have $$\left[\begin{array}{c}m+n \\
r
\end{array}\right]=\sum\limits_{0\leq t\leq n,r}v^{t(m+n)-nr}\left[\begin{array}{c}m \\
r-t
\end{array}\right]\left[\begin{array}{c}n \\
t
\end{array}\right].$$

$(b)$ Assume that $m\geq k \geq 0,~~\delta \in \mathbb{N}.$ We have
$$\sum\limits_{0\leq i\leq \delta}(-1)^{i}\left[\begin{array}{c}k+i-1 \\
i
\end{array}\right]\left[\begin{array}{c}m \\
\delta-i
\end{array}\right]v^{i(m-k)}=\left[\begin{array}{c}m-k \\
\delta
\end{array}\right]v^{-k\delta}.$$

$(c)$ Assume that $a\geq c\geq 0,~~u,r\in \mathbb{N},~~b\in
\mathbb{Z}.$ We have
$$\sum\limits_{0\leq f \leq u}(-1)^{f}\left[\begin{array}{c}a-c+f-1 \\
f
\end{array}\right]\left[\begin{array}{c}b+r-f \\
r
\end{array}\right]\left[\begin{array}{c}a+u \\
u-f
\end{array}\right]v^{f(u+c-r)}$$$$=\sum\limits_{0\leq \delta\leq u,r}v^{\delta b+(u-\delta)(c-a-r)}\left[\begin{array}{c}b+r-u \\
r-\delta
\end{array}\right]\left[\begin{array}{c}c+u-\delta \\
u-\delta
\end{array}\right]\left[\begin{array}{c}a+u \\
\delta
\end{array}\right].$$
\end{lemma}
\begin{proof}
$(a)$ is proved in [L5, 1.3 (e)].

$(b)$ is proved in [X1] and [X2].

Now we give the proof of $(c).$ We use induction on $a+u+r.$ When
$a=0,$ it follows from $(a)$. When $u=0,$ it is obvious. When $r=0,$
it follows from $(b).$

Now suppose that $a,u,r\geq 1.$ Note that $$\left[\begin{array}{c}a+u \\
u-f
\end{array}\right]=v^{-u+f}\left[\begin{array}{c}a+u-1 \\
u-f
\end{array}\right]+v^{a+f}\left[\begin{array}{c}a+u-1 \\
u-1-f
\end{array}\right],$$$$\left[\begin{array}{c}b+r-f \\
r
\end{array}\right]=v^{-r}\left[\begin{array}{c}b+r-f-1 \\
r
\end{array}\right]+v^{b-f}\left[\begin{array}{c}b+r-f-1 \\
r-1
\end{array}\right].$$\vskip2mm

We have\begin{align*}\hspace*{-0.6cm}&\sum\limits_{0\leq f \leq u}(-1)^{f}\left[\begin{array}{c}a-c+f-1 \\
f
\end{array}\right]\left[\begin{array}{c}b+r-f \\
r
\end{array}\right]\left[\begin{array}{c}a+u \\
u-f
\end{array}\right]v^{f(u+c-r)}\end{align*}\vskip-4mm
\begin{align*}
&=\sum\limits_{0\leq f \leq u}(-1)^{f}\hspace*{-0.1cm}\left[\begin{array}{c}a-c+f-1 \\
f
\end{array}\right]\hspace*{-0.1cm}\left[\begin{array}{c}b+r-1-f \\
r-1
\end{array}\right]\hspace*{-0.1cm}\left[\begin{array}{c}a+u-1 \\
u-1-f
\end{array}\right]\hspace*{-0.1cm}v^{f(u+c-r)+a+b}\\
&+\sum\limits_{0\leq f \leq u}(-1)^{f}\hspace*{-0.1cm}\left[\hspace*{-0.1cm}\begin{array}{c}a-c+f-1 \\
f
\end{array}\hspace*{-0.1cm}\right]\hspace*{-0.2cm}\left[\hspace*{-0.1cm}\begin{array}{c}b+r-1-f \\
r
\end{array}\hspace*{-0.1cm}\right]\hspace*{-0.2cm}\left[\hspace*{-0.1cm}\begin{array}{c}a+u-1 \\
u-1-f
\end{array}\hspace*{-0.1cm}\right]\hspace*{-0.1cm}v^{f(u+c-r)-r+a+f}\\&
+\sum\limits_{0\leq f \leq u}(-1)^{f}\hspace*{-0.1cm}\left[\begin{array}{c}a-c+f-1 \\
f
\end{array}\right]\hspace*{-0.1cm}\left[\begin{array}{c}b+r-f \\
r
\end{array}\right]\hspace*{-0.1cm}\left[\begin{array}{c}a+u-1 \\
u-f
\end{array}\right]v^{f(u+c-r)-u+f}
\end{align*}\vskip-4mm
\begin{align*}
&=\sum\limits_{0\leq f \leq u}(-1)^{f}\hspace*{-0.1cm}\left[\hspace*{-0.1cm}\begin{array}{c}a-c+f-1 \\
f
\end{array}\hspace*{-0.1cm}\right]\hspace*{-0.1cm}\left[\hspace*{-0.1cm}\begin{array}{c}b+r-1-f \\
r-1
\end{array}\hspace*{-0.1cm}\right]\hspace*{-0.1cm}\left[\hspace*{-0.1cm}\begin{array}{c}a+u-1 \\
u-1-f
\end{array}\hspace*{-0.1cm}\right]\hspace*{-0.1cm}v^{f(u+c-r)+a+b}\\&
+\sum\limits_{0\leq f \leq u}(-1)^{f}\hspace*{-0.1cm}\left[\hspace*{-0.1cm}\begin{array}{c}a-c+f-1 \\
f
\end{array}\hspace*{-0.1cm}\right]\hspace*{-0.1cm}\left[\hspace*{-0.1cm}\begin{array}{c}b+r-1-f \\
r
\end{array}\hspace*{-0.1cm}\right]\hspace*{-0.1cm}\left[\hspace*{-0.1cm}\begin{array}{c}a+u-1 \\
u-1-f
\end{array}\hspace*{-0.1cm}\right]\hspace*{-0.1cm}v^{f(u+c-r)-r+a+f}\\&
+\sum\limits_{0\leq f \leq u}(-1)^{f}\hspace*{-0.1cm}\left[\begin{array}{c}a-c+f-2 \\
f
\end{array}\right]\hspace*{-0.1cm}\left[\begin{array}{c}b+r-f \\
r
\end{array}\right]\hspace*{-0.1cm}\left[\begin{array}{c}a+u-1 \\
u-f
\end{array}\right]v^{f(u+c-r)-u}\\&
+\sum\limits_{0\leq f \leq u}(-1)^{f}\hspace*{-0.1cm}\left[\hspace*{-0.15cm}\begin{array}{c}a-c+f-2 \\
f-1
\end{array}\hspace*{-0.15cm}\right]\hspace*{-0.15cm}\left[\hspace*{-0.15cm}\begin{array}{c}b+r-f \\
r
\end{array}\hspace*{-0.15cm}\right]\hspace*{-0.15cm}\left[\begin{array}{c}a+u-1 \\
u-f
\end{array}\right]\hspace*{-0.1cm}v^{f(u+c-r)-u+f+a-c-1}.
\end{align*}

We easily see that the second summation and the fourth summation in the last expression of the above identities cancel
out. Using the induction hypothesis we see that the last expression of
the above identities is $$\sum\limits_{0\leq \delta\leq u-1,r-1}v^{\delta b+(u-1-\delta)(c-a+1-r)+a+b}\left[\hspace*{-0.15cm}\begin{array}{c}b+r-u \\
r-1-\delta
\end{array}\hspace*{-0.15cm}\right]\hspace*{-0.15cm}\left[\hspace*{-0.15cm}\begin{array}{c}c+u-1-\delta \\
u-1-\delta
\end{array}\hspace*{-0.15cm}\right]\hspace*{-0.15cm}\left[\hspace*{-0.15cm}\begin{array}{c}a+u-1 \\
\delta
\end{array}\hspace*{-0.15cm}\right]$$$$+\sum\limits_{0\leq \delta\leq u,r}v^{\delta b+(u-\delta)(c-a+1-r)-u}\left[\begin{array}{c}b+r-u \\
r-\delta
\end{array}\right]\left[\begin{array}{c}c+u-\delta \\
u-\delta
\end{array}\right]\left[\begin{array}{c}a+u-1 \\
\delta
\end{array}\right]$$$$=\sum\limits_{0\leq \delta\leq u,r}v^{\delta b+(u-\delta)(c-a-r)}\left[\hspace*{-0.15cm}\begin{array}{c}b+r-u \\
r-\delta
\end{array}\hspace*{-0.15cm}\right]\hspace*{-0.15cm}\left[\hspace*{-0.15cm}\begin{array}{c}c+u-\delta \\
u-\delta
\end{array}\hspace*{-0.15cm}\right]\hspace*{-0.15cm}\Bigg\{v^{-\delta}\left[\hspace*{-0.15cm}\begin{array}{c}a+u-1 \\
\delta
\end{array}\hspace*{-0.15cm}\right]+v^{a+u-\delta}$$$$\left[\hspace*{-0.15cm}\begin{array}{c}a+u-1 \\
\delta-1
\end{array}\hspace*{-0.15cm}\right]\Bigg\}=\sum\limits_{0\leq \delta\leq u,r}v^{\delta b+(u-\delta)(c-a-r)}\left[\hspace*{-0.15cm}\begin{array}{c}b+r-u \\
r-\delta
\end{array}\hspace*{-0.15cm}\right]\left[\hspace*{-0.15cm}\begin{array}{c}c+u-\delta \\
u-\delta
\end{array}\hspace*{-0.15cm}\right]\left[\hspace*{-0.15cm}\begin{array}{c}a+u \\
\delta
\end{array}\hspace*{-0.15cm}\right].$$
\end{proof}
\vskip2mm Now we start the proof of Theorem 3.1.\vskip2mm

\begin{proof}
Applying the anti-automorphism $\sigma$ to these elements listed in the sets (1), (2), $\ldots, (12)$ and (13) of Part (1.1), we exactly get these elements listed in $(1'),$ $(2'),$ $\ldots, (12')$ and $(13')$ of Part (1.2). If we can prove that these elements, which are listed in the sets (1), (2), $\ldots, (12)$ and (13) of Part (1.1), are elements of the canonical basis $\dot{\mathbf{B}}$, then we conclude that these elements
listed in $(1'),$ $(2'),$ $\ldots, (12')$ and $(13')$ are also elements of $\dot{\mathbf{B}}$ by Theorem 2.3. Thus it is sufficient to prove that these elements listed in the sets (1), (2), $\ldots, (12)$ and (13) of Part (1.1) are elements of $\dot{\mathbf{B}}$.

\vskip3mm
$(1)$ For the element
$e_{2}^{(h)}e_{1}^{(k)}e_{2}^{(j)}1_{(l,m)}f_{2}^{(u)}f_{1}^{(v)}f_{2}^{(w)},$
its image under the map
$\dot{\mathbf{U}}\rightarrow V(-s\omega_{1}-t\omega_{2})\otimes
V(a\omega_{1}+b\omega_{2}),$ which is given by $u\mapsto
u(\xi_{(-s,-t)}\otimes\eta_{(a,b)}),$ is zero unless $l+2v-(u+w)=a-s$ and $m+2(u+w)-v=b-t,$ in
which case we get\vskip4mm

$e_{2}^{(h)}e_{1}^{(k)}e_{2}^{(j)}1_{(l,m)}f_{2}^{(u)}f_{1}^{(v)}f_{2}^{(w)}(\xi_{(-s,-t)}\otimes\eta_{(a,b)})$
\vskip2mm(using 2.1 (g)) \vskip2mm
$=e_{2}^{(h)}e_{1}^{(k)}e_{2}^{(j)}(\xi_{(-s,-t)}\otimes
f_{2}^{(u)}f_{1}^{(v)}f_{2}^{(w)}\eta_{(a,b)})$\vskip2mm
(using 2.1 (a)) \vskip2mm
$=\sum\limits_{0\leq r \leq
j}e_{2}^{(h)}e_{1}^{(k)}v^{r(j-r-t)}(e_{2}^{(j-r)}\xi_{(-s,-t)}\otimes
e_{2}^{(r)}f_{2}^{(u)}f_{1}^{(v)}f_{2}^{(w)}\eta_{(a,b)})$\vskip2mm
(using 2.1 (c)) \vskip2mm
$=\sum\limits_{0\leq r \leq
j}v^{r(j-r-t)}e_{2}^{(h)}e_{1}^{(k)}\Bigg\{e_{2}^{(j-r)}\xi_{(-s,-t)}$\vskip2mm$\otimes
\Big\{\sum\limits_{0\leq d \leq u,r}f_{2}^{(u-d)}\left[\begin{array}{c}k_{2};2d-u-r\\
d
\end{array}\right]e_{2}^{(r-d)}f_{1}^{(v)}f_{2}^{(w)}\eta_{(a,b)}\Big\}\Bigg\}$\vskip2mm
(using 2.1 (c$-$d)) \vskip2mm
$=\sum\limits_{\substack{0\leq r \leq j\\0\leq d \leq
u,r}}v^{r(j-r-t)}e_{2}^{(h)}e_{1}^{(k)}\Bigg\{e_{2}^{(j-r)}\xi_{(-s,-t)}$\vskip2mm$\otimes
f_{2}^{(u-d)}\hspace*{-0.1cm}\left[\begin{array}{c}\hspace*{-0.1cm}k_{2};2d-u-r\hspace*{-0.1cm}\\
d
\end{array}\right]\hspace*{-0.1cm}f_{1}^{(v)}f_{2}^{(w-r+d)}\left[\begin{array}{c}k_{2};r-d-w\\
r-d
\end{array}\right]\eta_{(a,b)}\Bigg\}$\vskip2mm
(using 2.1 (e))\vskip2mm
$=\sum\limits_{\substack{0\leq r \leq j\\0\leq d \leq
u,r}}\hspace*{-0.15cm}v^{r(j-r-t)}e_{2}^{(h)}e_{1}^{(k)}\hspace*{-0.05cm}
\Bigg\{\hspace*{-0.1cm}e_{2}^{(j-r)}\xi_{(-s,-t)}$\vskip2mm$\otimes
\left[\hspace*{-0.1cm}\begin{array}{c}b+r-u-2w+v \\
d
\end{array}\hspace*{-0.1cm}\right]\hspace*{-0.1cm}\left[\hspace*{-0.1cm}\begin{array}{c}b+r-d-w \\
r-d
\end{array}\hspace*{-0.1cm}\right]f_{2}^{(u-d)}f_{1}^{(v)}f_{2}^{(w-r+d)}\eta_{(a,b)}\Bigg\}$\vskip2mm
$=\sum\limits_{\substack{0\leq r \leq j\\0\leq d
\leq u,r}}v^{r(j-r-t)}\hspace*{-0.1cm}\left[\hspace*{-0.1cm}\begin{array}{c}b+r-u-2w+v \\
d
\end{array}\hspace*{-0.1cm}\right]\hspace*{-0.1cm}\left[\hspace*{-0.1cm}\begin{array}{c}b+r-d-w \\
r-d
\end{array}\hspace*{-0.1cm}\right]\hspace*{-0.1cm}e_{2}^{(h)}
\Bigg\{\hspace*{-0.1cm}\sum\limits_{0\leq p\leq k}$\vskip2mm$\times
v^{p(k-p)+p(-s-j+r)}e_{1}^{(k-p)}e_{2}^{(j-r)}\xi_{(-s,-t)}\otimes
f_{2}^{(u-d)}e_{1}^{(p)}f_{1}^{(v)}f_{2}^{(w-r+d)}\eta_{(a,b)}\hspace*{-0.1cm}\Bigg\}$\vskip2mm$
=\sum\limits_{\substack{0\leq r \leq j\\0\leq d \leq u,r\\0\leq
p\leq k,v}}v^{r(j-r-t)+p(k-p-s-j+r)}\hspace*{-0.1cm}\left[\hspace*{-0.1cm}\begin{array}{c}b+r-u-2w+v \\
d
\end{array}\hspace*{-0.1cm}\right]\hspace*{-0.15cm}\left[\hspace*{-0.1cm}\begin{array}{c}b+r-d-w \\
r-d
\end{array}\hspace*{-0.1cm}\right]$\vskip2mm$\times e_{2}^{(h)}\Bigg\{e_{1}^{(k-p)} e_{2}^{(j-r)}\xi_{(-s,-t)}\otimes
f_{2}^{(u-d)}f_{1}^{(v-p)}\left[\begin{array}{c}k_1;p-v \\
p
\end{array}\right]f_{2}^{(w-r+d)}\eta_{(a,b)}\Bigg\}$\vskip2mm$
=\sum\limits_{\substack{0\leq
p\leq k,v\\0\leq r \leq j\\0\leq d \leq u,r}}v^{r(j-r-t)+p(k-p-s-j+r)}$\vskip2mm
$\times\left[\hspace*{-0.1cm}\begin{array}{c}b+r-u-2w+v \\
d
\end{array}\hspace*{-0.1cm}\right]\hspace*{-0.1cm}\left[\hspace*{-0.1cm}\begin{array}{c}b+r-d-w \\
r-d
\end{array}\hspace*{-0.1cm}\right]\hspace*{-0.1cm}\left[\hspace*{-0.1cm}\begin{array}{c}a+p-v+w-r+d\\
p
\end{array}\hspace*{-0.1cm}\right]$\vskip2mm$\times e_{2}^{(h)}
\Bigg\{e_{1}^{(k-p)}e_{2}^{(j-r)}\xi_{(-s,-t)}\otimes
f_{2}^{(u-d)}f_{1}^{(v-p)}f_{2}^{(w-r+d)}\eta_{(a,b)}\Bigg\}$\vskip2mm
$=\sum\limits_{\substack{0\leq
p\leq k,v\\0\leq r \leq j\\0\leq d \leq u,r}}v^{r(j-r-t)+p(k-p-s-j+r)}\hspace*{-0.1cm}\left[\hspace*{-0.1cm}\begin{array}{c}b+r-u-2w+v \\
d
\end{array}\hspace*{-0.1cm}\right]\hspace*{-0.13cm}\left[\hspace*{-0.1cm}\begin{array}{c}b+r-d-w \\
r-d
\end{array}\hspace*{-0.1cm}\right]$\vskip2mm$
\times\left[\hspace*{-0.1cm}\begin{array}{c}a+p-v+w-r+d\\
p
\end{array}\hspace*{-0.1cm}\right]\hspace*{-0.1cm}
\Bigg\{\hspace*{-0.1cm}\sum\limits_{0\leq q\leq
h}v^{q(h-q)+q(-t+2j-2r-k+p)} $\vskip2mm$\times
e_{2}^{(h-q)}e_{1}^{(k-p)}e_{2}^{(j-r)}\xi_{(-s,-t)}\otimes
e_{2}^{(q)}f_{2}^{(u-d)}f_{1}^{(v-p)}f_{2}^{(w-r+d)}\eta_{(a,b)}\Bigg\}$\vskip2mm$
=\sum\limits_{\substack{0\leq p\leq k,v\\0\leq d \leq u,r\\0\leq r
\leq j,~0\leq q\leq
h}}\hspace*{-0.2cm}v^{r(j-r-t)+p(k-p-s-j+r)+q(h-q-t+2j-2r-k+p)}
$\vskip2mm$\times\left[\hspace*{-0.1cm}\begin{array}{c}b+r-u-2w+v \\
d
\end{array}\hspace*{-0.1cm}\right]\hspace*{-0.1cm}\left[\hspace*{-0.1cm}\begin{array}{c}b+r-d-w \\
r-d
\end{array}\hspace*{-0.1cm}\right]\hspace*{-0.1cm}\left[\hspace*{-0.1cm}\begin{array}{c}a+p-v+w-r+d\\
p
\end{array}\hspace*{-0.1cm}\right]$\vskip2mm$\times e_{2}^{(h-q)}e_{1}^{(k-p)}e_{2}^{(j-r)}\xi_{(-s,-t)}\otimes
\Bigg\{\sum\limits_{0\leq t'\leq q,u-d}f_{2}^{(u-d-t')}$\vskip2mm$\times\left[\begin{array}{c}k_{2};2t'-(u-d+q) \\
t'
\end{array}\right]e_{2}^{(q-t')}f_{1}^{(v-p)}f_{2}^{(w-r+d)}\eta_{(a,b)}\Bigg\}$\vskip2mm
$=\sum\limits_{\substack{0\leq p \leq k,v \\0\leq r \leq j,~0\leq d
\leq u,r\\ 0\leq q\leq h,~0\leq t'\leq
q,u-d}}v^{r(j-r-t)+p(k-p-s-j+r)+q(h-q-t-k+p+2j-2r)}$\vskip2mm$
\times\left[\hspace*{-0.1cm}\begin{array}{c}b+r-u-2w+v \\
d
\end{array}\hspace*{-0.1cm}\right]\hspace*{-0.1cm}\hspace*{-0.1cm}\left[\hspace*{-0.1cm}\begin{array}{c}b+r-d-w \\
r-d
\end{array}\hspace*{-0.1cm}\right]\hspace*{-0.1cm}\left[\hspace*{-0.1cm}\begin{array}{c}a+p-v+w-r+d\\
p
\end{array}\hspace*{-0.1cm}\right]$\vskip2mm$\times
e_{2}^{(h-q)}e_{1}^{(k-p)}e_{2}^{(j-r)}\xi_{(-s,-t)}\otimes
f_{2}^{(u-d-t')}\left[\hspace*{-0.1cm}\begin{array}{c}k_{2};2t'-(u-d+q) \\
t'
\end{array}\hspace*{-0.1cm}\right]$\vskip2mm$\times f_{1}^{(v-p)}f_{2}^{(w-r+d-q+t')}
\hspace*{-0.1cm}\left[\hspace*{-0.15cm}\begin{array}{c}k_{2};q-t'-w+r-d \\
q-t'
\end{array}\hspace*{-0.15cm}\right]\hspace*{-0.1cm}\eta_{(a,b)}$\vskip2mm
$=\sum\limits_{\substack{0\leq p \leq k,v \\0\leq r \leq j,~0\leq d
\leq u,r\\ 0\leq q\leq h,~0\leq t'\leq
q,u-d}}v^{r(j-r-t)+p(k-p-s-j+r)+q(h-q-t-k+p+2j-2r)}
$\vskip2mm$\times\left[\hspace*{-0.1cm}\begin{array}{c}b+r-u-2w+v \\
d
\end{array}\hspace*{-0.1cm}\right]
\hspace*{-0.15cm}\left[\hspace*{-0.1cm}\begin{array}{c}b+r-d-w \\
r-d
\end{array}\hspace*{-0.1cm}\right]
\hspace*{-0.15cm}\left[\hspace*{-0.15cm}\begin{array}{c}a+p-v+w-r+d \\
p
\end{array}\hspace*{-0.15cm}\right]$\vskip2mm$\times\left[\hspace*{-0.1cm}\begin{array}{c}b+q-u-2w+v-p+2r-d \\
t'
\end{array}\hspace*{-0.1cm}\right]\hspace*{-0.15cm}\left[\hspace*{-0.15cm}\begin{array}{c}b+q-t'-w+r-d \\
q-t'
\end{array}\hspace*{-0.15cm}\right]$\vskip2mm$\times e_{2}^{(h-q)}e_{1}^{(k-p)}e_{2}^{(j-r)}\xi_{(-s,-t)}\otimes
f_{2}^{(u-d-t')}f_{1}^{(v-p)}f_{2}^{(w-r+d-q+t')}\eta_{(a,b)}.$
\vskip4mm

Let $A$ denote the degree of the coefficient (with respect to $v$) in the last expression of the above identities, then\vskip2mm
$A=r(j-r-t)+p(k-p-s-j+r)+q(h-q-t-k+p+2j-2r)+d(b+r-d-u-2w+v)+(r-d)(b-w)+
p(a-v+w-r+d)+t'(b+v+q-u-2w-p-t'+2r-d)+(q-t')(b-w+r-d).$

\vskip2mm If $-l \geq v+k-j-u,~~-m\geq u+j,~~k\geq h+j,~~v\geq u+w$,
then we get $v\geq a+w+k-s-j,~~u+2w\geq b+v+j-t,~~w\geq b+j-t,~~k\geq h+j,~~v\geq u+w.$
\vskip2mm In this case $A\leq
-p^2+pd+p(a+w-v+k-s-j)+q(h+j-k)+q(-q+p-t+j-2r)+
t'(q-t'-j+t+2r-d-p)+(q-t')(t-j+r-d)+r(j-r-t)+d(r-d+t-j)+(r-d)(t-j)\leq
-p^2+pd-r^2+d(r-d)+p(a+w-v+k-s-j)+q(h+j-k)+q(-q+p-r-d)+t'(q-t'+r-p)
\leq -p^2+pd-r^2+d(r-d)+q(-q+p-r-d)+t'(q-t'+r-p).$ \vskip2mm

If $q\geq p,$ then $A\leq
-(p-d)^2-pd+r(d-r)+(t'-q)(q-p+r)-t'^2-qd\leq 0;$ \vskip2mm If $q<
p,$ then $A\leq -p^2+pd-r^2+d(r-d)+q(-q+p-r-d)+t'(q-t'+r-p)
\leq$\vskip2mm$\begin{cases}-(p-d)^2-pd+r(d-r)-t'^2+(q-p)t'+
r(t'-q)-q^2+q(p-d)< 0 &\\\hfill \hbox {if }p\leq d;         \\
-t'^2+(q-p)t'+r(t'-q)-q^2+(p-q)(d-p)+r(d-r)-d^2 < 0&
\\\hfill\hbox {if } p > d.
\end{cases}$\vskip4mm

So we get $A\leq 0$  and  $A=0$ if and only if $p=r=d=q=t'=0.$
Meanwhile the above expression is fixed by the involution $\Psi$ of
$V(-s\omega_{1}-t\omega_{2})\otimes V(a\omega_{1}+b\omega_{2})$,
since the element
$e_{2}^{(h)}e_{1}^{(k)}e_{2}^{(j)}1_{(l,m)}f_{2}^{(u)}f_{1}^{(v)}f_{2}^{(w)}$
is fixed by $-:\dot{\mathbf{U}}\rightarrow \dot{\mathbf{U}}$. By
using the definitions, we can see that the above expression is
$(\theta_{2}^{(h)}\theta_{1}^{(k)}\theta_{2}^{(j)}\lozenge
\theta_{2}^{(u)}\theta_{1}^{(v)}\theta_{2}^{(w)})_{(s,t),(a,b)},$ if
$v\geq a+w+k-s-j,~~u+2w\geq b+v+j-t,~~w\geq b+j-t,~~k\geq h+j,~~v\geq u+w.$ Hence the
element listed in $(1)$ is an element of $\dot{\mathbf{B}}.$

\vskip3mm
$(2)$ We want to compute the image of this element listed in $(2)$ under the map
$\dot{\mathbf{U}}\rightarrow V(-s\omega_{1}-t\omega_{2})\otimes
V(a\omega_{1}+b\omega_{2}),$ which is given by $u\mapsto
u(\xi_{(-s,-t)}\otimes\eta_{(a,b)}).$ Its image is zero unless
$l+2v-(u+w)=a-s$ and $m+2(u+w)-v=b-t.$ In this case, using the computations in the
proof of (1), we get the image of this element under
the above-mentioned map, which is the following element (replacing $p$ by $l$):\vskip3mm

$B=\sum\limits_{0\leq l \leq j,u}\hspace*{-0.1cm}(-1)^{l}\hspace*{-0.1cm}\left[\hspace*{-0.2cm}\begin{array}{c}v-u-2w+b+j-t+l-1 \\
l
\end{array}\hspace*{-0.2cm}\right]$\vskip2mm$\times\Bigg\{\sum\limits_{\substack{0\leq p \leq k,v \\0\leq r \leq j-l,~0\leq d
\leq u-l,r\\ 0\leq q\leq h,~0\leq t'\leq
q,u-l-d}}v^{r(j-l-r-t)+p(k-p-s-j+l+r)+q(h-q-t-k+p+2j-2l-2r)}
$\vskip2mm$\times\left[\hspace*{-0.1cm}\begin{array}{c}b+r-u-2w+v+l \\
d
\end{array}\hspace*{-0.1cm}\right]
\hspace*{-0.15cm}\left[\hspace*{-0.1cm}\begin{array}{c}b+r-d-w \\
r-d
\end{array}\hspace*{-0.1cm}\right]
\hspace*{-0.15cm}\left[\hspace*{-0.15cm}\begin{array}{c}a+p-v+w-r+d \\
p
\end{array}\hspace*{-0.15cm}\right]$\vskip2mm$\times\left[\hspace*{-0.1cm}\begin{array}{c}b+q-u-2w+v-p+2r-d+l \\
t'
\end{array}\hspace*{-0.1cm}\right]\hspace*{-0.15cm}\left[\hspace*{-0.15cm}\begin{array}{c}b+q-t'-w+r-d \\
q-t'
\end{array}\hspace*{-0.15cm}\right]$\vskip2mm$\times e_{2}^{(h-q)}e_{1}^{(k-p)}e_{2}^{(j-l-r)}\xi_{(-s,-t)}\otimes
f_{2}^{(u-l-d-t')}f_{1}^{(v-p)}f_{2}^{(w-r+d-q+t')}\eta_{(a,b)}\Bigg\}.$

\vskip3mm If $-l \geq v-u+k-j,~~u+j+(u+w-v)\leq -m \leq u+j,~~-m \geq
u+j+(j+h-k),~~k\geq h+j,~~v\geq u+w$, by the equalities $l+2v-(u+w)=a-s$ and $m+2(u+w)-v=b-t,$ then we get $v\geq a+w+k-s-j, ~~b+v+j-t+(j+h-k)\leq u+2w\leq b+v+j-t, ~~w\geq b+j-t,~~k\geq
h+j,~~ v\geq u+w.$ Under these conditions we have the following.\vskip3mm

If we let $r=r'-l,~~ d=d'-l,$ then we get\vskip3mm

$B=\sum\limits_{0\leq l \leq j,u}\hspace*{-0.1cm}(-1)^{l}\hspace*{-0.1cm}\left[\hspace*{-0.2cm}\begin{array}{c}v-u-2w+b+j-t+l-1 \\
l
\end{array}\hspace*{-0.2cm}\right]$\vskip2mm$\times\Bigg\{\sum\limits_{\substack{0\leq p \leq k,v \\l\leq r' \leq j,~l\leq d'
\leq u,r'\\ 0\leq q\leq h,~0\leq t'\leq
q,u-d'}}v^{(r'-l)(j-r'-t)+p(k-p-s-j+r')+q(h-q-t-k+p+2j-2r')}
$\vskip2mm$\times\left[\hspace*{-0.1cm}\begin{array}{c}b-u-2w+v+r' \\
d'-l
\end{array}\hspace*{-0.1cm}\right]
\hspace*{-0.15cm}\left[\hspace*{-0.1cm}\begin{array}{c}b+r'-d'-w \\
r'-d'
\end{array}\hspace*{-0.1cm}\right]
\hspace*{-0.15cm}\left[\hspace*{-0.15cm}\begin{array}{c}a+p-v+w-r'+d' \\
p
\end{array}\hspace*{-0.15cm}\right]$\vskip2mm$\times\left[\hspace*{-0.1cm}\begin{array}{c}b+q-u-2w+v-p+2r'-d' \\
t'
\end{array}\hspace*{-0.1cm}\right]\hspace*{-0.15cm}\left[\hspace*{-0.15cm}\begin{array}{c}b+q-t'-w+r'-d' \\
q-t'
\end{array}\hspace*{-0.15cm}\right]$\vskip2mm$\times e_{2}^{(h-q)}e_{1}^{(k-p)}e_{2}^{(j-r')}\xi_{(-s,-t)}\otimes
f_{2}^{(u-d'-t')}f_{1}^{(v-p)}f_{2}^{(w-r'+d'-q+t')}\eta_{(a,b)}\Bigg\}$

\vskip3mm $=\sum\limits_{\substack{0\leq p \leq k,v \\0\leq r' \leq j,~0\leq d'
\leq u,r'\\ 0\leq q\leq h,~0\leq t'\leq
q,u-d'}}v^{r'(j-r'-t)+p(k-p-s-j+r')+q(h-q-t-k+p+2j-2r')}
$\vskip2mm$\times\left[\hspace*{-0.1cm}\begin{array}{c}b+r'-d'-w \\
r'-d'
\end{array}\hspace*{-0.1cm}\right]
\hspace*{-0.15cm}\left[\hspace*{-0.15cm}\begin{array}{c}a+p-v+w-r'+d' \\
p
\end{array}\hspace*{-0.15cm}\right]$\vskip2mm$\times\left[\hspace*{-0.1cm}\begin{array}{c}b+q-u-2w+v-p+2r'-d' \\
t'
\end{array}\hspace*{-0.1cm}\right]\hspace*{-0.15cm}\left[\hspace*{-0.15cm}\begin{array}{c}b+q-t'-w+r'-d' \\
q-t'
\end{array}\hspace*{-0.15cm}\right]$\vskip3mm$\times
\Bigg\{\sum\limits_{0\leq l \leq d'}\hspace*{-0.1cm}(-1)^{l}\hspace*{-0.1cm}\left[\hspace*{-0.2cm}\begin{array}{c}v-u-2w+b+j-t+l-1 \\
l
\end{array}\hspace*{-0.2cm}\right]\hspace*{-0.15cm}\left[\hspace*{-0.1cm}\begin{array}{c}b-u-2w+v+r' \\
d'-l
\end{array}\hspace*{-0.1cm}\right]
\hspace*{-0.15cm}v^{l(r'+t-j)}\Bigg\}$\vskip3mm$
\times e_{2}^{(h-q)}e_{1}^{(k-p)}e_{2}^{(j-r')}\xi_{(-s,-t)}\otimes
f_{2}^{(u-d'-t')}f_{1}^{(v-p)}f_{2}^{(w-r'+d'-q+t')}\eta_{(a,b)}.$

\vskip3mm Using Lemma 4.1 (b), we see that \vskip3mm$\sum\limits_{0\leq l \leq d'}\hspace*{-0.1cm}(-1)^{l}\hspace*{-0.1cm}\left[\hspace*{-0.2cm}\begin{array}{c}v-u-2w+b+j-t+l-1 \\
l
\end{array}\hspace*{-0.2cm}\right]\hspace*{-0.15cm}\left[\hspace*{-0.1cm}\begin{array}{c}b-u-2w+v+r' \\
d'-l
\end{array}\hspace*{-0.1cm}\right]
\hspace*{-0.15cm}v^{l(r'+t-j)}$\vskip2mm$=\left[\hspace*{-0.1cm}\begin{array}{c}r'+t-j \\
d'
\end{array}\hspace*{-0.1cm}\right]v^{-(v-u-2w+b+j-t)d'}.$

\vskip3mm So we have \vskip3mm$B=\sum\limits_{\substack{0\leq p \leq k,v \\0\leq r' \leq j,~0\leq d'
\leq u,r'\\ 0\leq q\leq h,~0\leq t'\leq
q,u-d'}}v^{-(v-u-2w+b+j-t)d'+r'(j-r'-t)+p(k-p-s-j+r')+q(h-q-t-k+p+2j-2r')}
$\vskip2mm$\times\left[\hspace*{-0.1cm}\begin{array}{c}r'+t-j \\
d'
\end{array}\hspace*{-0.1cm}\right]\hspace*{-0.15cm}\left[\hspace*{-0.1cm}\begin{array}{c}b+r'-d'-w \\
r'-d'
\end{array}\hspace*{-0.1cm}\right]
\hspace*{-0.15cm}\left[\hspace*{-0.15cm}\begin{array}{c}a+p-v+w-r'+d' \\
p
\end{array}\hspace*{-0.15cm}\right]$\vskip2mm$\times\left[\hspace*{-0.1cm}\begin{array}{c}b+q-u-2w+v-p+2r'-d' \\
t'
\end{array}\hspace*{-0.1cm}\right]\hspace*{-0.15cm}\left[\hspace*{-0.15cm}\begin{array}{c}b+q-t'-w+r'-d' \\
q-t'
\end{array}\hspace*{-0.15cm}\right]$\vskip2mm$
\times e_{2}^{(h-q)}e_{1}^{(k-p)}e_{2}^{(j-r')}\xi_{(-s,-t)}\otimes
f_{2}^{(u-d'-t')}f_{1}^{(v-p)}f_{2}^{(w-r'+d'-q+t')}\eta_{(a,b)}.$

\vskip3mm Let $\mathbb{B}$ denote the degree of the coefficient in
the above expression $B,$ then we have \vskip2mm
$\mathbb{B}=-(v-u-2w+b+j-t)d'+r'(j-r'-t)+p(k-p-s-j+r')+q(h-q-t-k+p+2j-2r')+d'(r'-d'+t-j)+
(r'-d')(b-w)+p(a-v+w-r'+d')+t'(b+q-p-u-2w+v-t'+2r'-d')+(q-t')(b-w+r'-d').$

\vskip3mm Since $v\geq a+w+k-s-j, ~~b+v+j-t+(j+h-k)\leq u+2w\leq b+v+j-t, ~~w\geq b+j-t,~~k\geq
h+j,~~ v\geq u+w,$ in this case we get
\vskip3mm
$\mathbb{B}\leq p(-p+d')+r'(j-r'-t)+q(-q-t+p+j-2r')+q(h+j-k)+d'(r'-d'+t-j)+(r'-d')(b-w)+t'(q-p-t'+2r'-d'+t-j)+(q-t')(b-w+r'-d')\leq -p^2+pd'-r'^2+d'(r'-d')+q(-q+p-r'-d')+t'(q-t'+r'-p)\leq 0$ and $\mathbb{B}=0$ if and only if $p=q=r'=d'=t'=0.$\vskip3mm
Meanwhile this element $B$ is fixed by the involution $\Psi$ of
$V(-s\omega_{1}-t\omega_{2})\otimes V(a\omega_{1}+b\omega_{2})$,
since the element listed in (2) is fixed by
$-:\dot{\mathbf{U}}\rightarrow \dot{\mathbf{U}}$. By using the
definitions, we can see that this element $B$ is
$(\theta_{2}^{(h)}\theta_{1}^{(k)}\theta_{2}^{(j)}\lozenge
\theta_{2}^{(u)}\theta_{1}^{(v)}\theta_{2}^{(w)})_{(s,t),(a,b)},$
$\mbox{if}~v\geq a+w+k-s-j, ~~b+v+j-t+(j+h-k)\leq u+2w\leq b+v+j-t, ~~w\geq b+j-t,~~k\geq
h+j,~~ v\geq u+w.$ Hence the
element listed in (2) is an element of $\dot{\mathbf{B}}.$

\vskip3mm Similarly we can deal with $(3)$-$(5)$ by using the computations in the proof of (1) and repeatedly using Lemma 4.1 (b) (referring to the proof of (6) below).

\vskip3mm

$(6)$ We want to compute the image of this element listed in $(6)$ under the map
$\dot{\mathbf{U}}\rightarrow V(-s\omega_{1}-t\omega_{2})\otimes
V(a\omega_{1}+b\omega_{2}),$ which is given by $u\mapsto
u(\xi_{(-s,-t)}\otimes\eta_{(a,b)}).$ Its image is zero unless
$l+2v-(u+w)=a-s$ and $m+2(u+w)-v=b-t.$ In this case, using the computations in the proof of (1), we get the image of this element under
the above-mentioned map, which is the following element (replacing $p, q, r,
i$ by $l, f, e, z$ respectively):\vskip3mm
$C=\sum\limits_{\substack{0\leq l,f,~ l+f \leq j\\0\leq e,z,~
e+z\leq h\\0\leq l,e,~ l+e\leq u\\0\leq f,z, ~f+z\leq
w}}\hspace*{-0.1cm}(-1)^{l+f+e+z}\hspace*{-0.1cm}\left[\hspace*{-0.2cm}\begin{array}{c}v-u-2w+b+j-t+2z+f+e+l-1 \\
l
\end{array}\hspace*{-0.2cm}\right]$\vskip2mm$\times\left[\hspace*{-0.2cm}\begin{array}{c}b-w+j-t+z+f-1 \\
f
\end{array}\hspace*{-0.2cm}\right]\hspace*{-0.1cm}\left[\hspace*{-0.2cm}\begin{array}{c}b-w+j-t+j+h-k+z-1 \\
z
\end{array}\hspace*{-0.2cm}\right]$\vskip2mm$\times\left[\hspace*{-0.2cm}\begin{array}{c}v-u-2w+b+j-t+j+h-k+z+e-1 \\
e
\end{array}\hspace*{-0.2cm}\right]$\vskip2mm$\times\Bigg\{\sum\limits_{\substack{0\leq p \leq k,v
\\0\leq r \leq j-l-f\\0\leq d \leq u-l-e,r\\ 0\leq q\leq h-e-z\\0\leq
t'\leq
q,u-l-e-d}}\left[\hspace*{-0.2cm}\begin{array}{c}b+q-u-2w+v-p+2r-d+l+e+2f+2z \\
t'
\end{array}\hspace*{-0.2cm}\right]$\vskip2mm$\times v^{r(j-l-f-r-t)+p(k-p-s-j+l+f+r)+q(h-e-z-q-t-k+p+2j-2l-2f-2r)}
$\vskip2mm$\times\left[\hspace*{-0.1cm}\begin{array}{c}b+r-u-2w+v+l+e+2f+2z \\
d
\end{array}\hspace*{-0.1cm}\right]
\hspace*{-0.15cm}\left[\hspace*{-0.1cm}\begin{array}{c}b+r-d-w+f+z \\
r-d
\end{array}\hspace*{-0.1cm}\right]
$\vskip2mm$\times\left[\hspace*{-0.1cm}\begin{array}{c}a+p-v+w-r+d-f-z \\
p
\end{array}\hspace*{-0.1cm}\right]\hspace*{-0.15cm}\left[\hspace*{-0.15cm}\begin{array}{c}b+q-t'-w+r-d+f+z \\
q-t'
\end{array}\hspace*{-0.15cm}\right]$\vskip2mm$\times e_{2}^{(h-e-z-q)}e_{1}^{(k-p)}e_{2}^{(j-l-f-r)}\xi_{(-s,-t)}\hspace*{-0.1cm}\otimes\hspace*{-0.1cm}
f_{2}^{(u-l-e-d-t')}f_{1}^{(v-p)}f_{2}^{(w-f-z-r+d-q+t')}\eta_{(a,b)}\Bigg\}.$

\vskip3mm If
$-m\leq u+j+(j+h-k)+(u+w-v),~~-l-m\geq j+h+u+w,~~k\geq h+j,~~ v\geq
u+w,$ by the equalities $l+2v-(u+w)=a-s$ and $m+2(u+w)-v=b-t,$ then we get $w\leq b+j+(j+h-k)-t, ~~v\geq a+b+j+h-s-t,~~k\geq
h+j,~~ v\geq u+w,$ from them we also get $v\geq a+w+k-s-j, ~~u+2w\leq
b+v+j-t+j+h-k\leq b+v+j-t.$ Under these conditions we have the following.\vskip3mm

If we let $q=q'-e-z,~~ r=r'-l-f,~~ d=d'-l,~~ t'=t''-e,$ then we get

\vskip3mm $C=\sum\limits_{\substack{0\leq l,f,~ l+f \leq j\\0\leq
e,z,~ e+z\leq h\\0\leq l,e,~ l+e\leq u\\0\leq f,z, ~f+z\leq
w}}\hspace*{-0.1cm}(-1)^{l+f+e+z}\hspace*{-0.1cm}\left[\hspace*{-0.2cm}\begin{array}{c}v-u-2w+b+j-t+2z+f+e+l-1 \\
l
\end{array}\hspace*{-0.2cm}\right]$\vskip2mm$\times\left[\hspace*{-0.2cm}\begin{array}{c}b-w+j-t+z+f-1 \\
f
\end{array}\hspace*{-0.2cm}\right]\hspace*{-0.1cm}\left[\hspace*{-0.2cm}\begin{array}{c}b-w+j-t+j+h-k+z-1 \\
z
\end{array}\hspace*{-0.2cm}\right]$\vskip2mm$\times\left[\hspace*{-0.2cm}\begin{array}{c}v-u-2w+b+j-t+j+h-k+z+e-1 \\
e
\end{array}\hspace*{-0.2cm}\right]$\vskip2mm$\times\Bigg\{\sum\limits_{\substack{0\leq p \leq k,v
\\l+f\leq r' \leq j,~l\leq d' \leq u-e,r'-f\\ e+z\leq q'\leq h,~e\leq
t''\leq
q'-z,u-d'}}\hspace*{-0.1cm}\left[\hspace*{-0.2cm}\begin{array}{c}b+q'-u-2w+v-p+2r'-d'+z \\
t''-e
\end{array}\hspace*{-0.2cm}\right]$\vskip2mm$\times v^{(r'-l-f)(j-r'-t)+p(k-p-s-j+r')+(q'-e-z)(h-q'-t-k+p+2j-2r')}
$\vskip2mm$\times\left[\hspace*{-0.1cm}\begin{array}{c}b+r-u-2w+v+e+f+2z \\
d'-l
\end{array}\hspace*{-0.1cm}\right]
\hspace*{-0.15cm}\left[\hspace*{-0.1cm}\begin{array}{c}b+r'-d'-w+z \\
r'-d'-f
\end{array}\hspace*{-0.1cm}\right]
$\vskip2mm$\times\left[\hspace*{-0.1cm}\begin{array}{c}a+p-v+w-r'+d'-z \\
p
\end{array}\hspace*{-0.1cm}\right]\hspace*{-0.15cm}\left[\hspace*{-0.15cm}\begin{array}{c}b+q'-t''-w+r'-d' \\
q'-t''-z
\end{array}\hspace*{-0.15cm}\right]$\vskip2mm$\times e_{2}^{(h-q')}e_{1}^{(k-p)}e_{2}^{(j-r')}\xi_{(-s,-t)}\hspace*{-0.1cm}\otimes\hspace*{-0.1cm}
f_{2}^{(u-d'-t'')}f_{1}^{(v-p)}f_{2}^{(w-r'+d'-q'+t'')}\eta_{(a,b)}\Bigg\}$

\vskip3mm $=\sum\limits_{\substack{0\leq f \leq j,~0\leq e\leq
u\\0\leq e,z,~
e+z\leq h\\0\leq f,z, ~f+z\leq w}}\hspace*{-0.05cm}\left[\hspace*{-0.2cm}\begin{array}{c}v-u-2w+b+j-t+j+h-k+z+e-1 \\
e
\end{array}\hspace*{-0.2cm}\right]$\vskip2mm$\times (-1)^{f+e+z}\hspace*{-0.15cm}\left[\hspace*{-0.2cm}\begin{array}{c}b\hspace*{-0.02cm}-\hspace*{-0.02cm}w\hspace*{-0.02cm}+\hspace*{-0.02cm}j\hspace*{-0.02cm}-\hspace*{-0.02cm}t\hspace*{-0.02cm}+\hspace*{-0.02cm}z\hspace*{-0.02cm}+\hspace*{-0.02cm}f\hspace*{-0.02cm}-\hspace*{-0.02cm}1 \\
f
\end{array}\hspace*{-0.2cm}\right]\hspace*{-0.15cm}\left[\hspace*{-0.2cm}\begin{array}{c}b\hspace*{-0.02cm}-\hspace*{-0.02cm}w\hspace*{-0.02cm}+\hspace*{-0.02cm}j\hspace*{-0.02cm}-\hspace*{-0.02cm}t\hspace*{-0.02cm}+\hspace*{-0.02cm}j\hspace*{-0.02cm}+\hspace*{-0.02cm}h\hspace*{-0.02cm}-\hspace*{-0.02cm}k\hspace*{-0.02cm}+\hspace*{-0.02cm}z\hspace*{-0.02cm}-\hspace*{-0.02cm}1 \\
z
\end{array}\hspace*{-0.2cm}\right]$\vskip2mm$\times\Bigg\{\sum\limits_{\substack{0\leq p \leq k,v
\\f\leq r' \leq j,~0\leq d' \leq u-e,r'-f\\ e+z\leq q'\leq h,~e\leq
t''\leq
q'-z,u-d'}}\hspace*{-0.1cm}\left[\hspace*{-0.2cm}\begin{array}{c}b+q'-u-2w+v-p+2r'-d'+z \\
t''-e
\end{array}\hspace*{-0.2cm}\right]$\vskip2mm$\times v^{(r'-f)(j-r'-t)+p(k-p-s-j+r')+(q'-e-z)(h-q'-t-k+p+2j-2r')}
\hspace*{-0.15cm}\left[\hspace*{-0.2cm}\begin{array}{c}b\hspace*{-0.02cm}+\hspace*{-0.02cm}r'\hspace*{-0.02cm}-\hspace*{-0.02cm}d'\hspace*{-0.02cm}-\hspace*{-0.02cm}w\hspace*{-0.02cm}+\hspace*{-0.02cm}z \\
r'-d'-f
\end{array}\hspace*{-0.2cm}\right]
$\vskip2mm$\times\left[\hspace*{-0.1cm}\begin{array}{c}a+p-v+w-r'+d'-z \\
p
\end{array}\hspace*{-0.1cm}\right]\hspace*{-0.15cm}\left[\hspace*{-0.15cm}\begin{array}{c}b+q'-t''-w+r'-d' \\
q'-t''-z
\end{array}\hspace*{-0.15cm}\right]$\vskip2mm$\times\Bigg\{\hspace*{-0.15cm}
\sum\limits_{0\leq l\leq d'}(-1)^{l}\hspace*{-0.1cm}\left[\hspace*{-0.2cm}\begin{array}{c}v-u-2w+b+j-t+2z+f+e+l-1 \\
l
\end{array}\hspace*{-0.2cm}\right]$\vskip2mm$\times\left[\hspace*{-0.1cm}\begin{array}{c}b+r-u-2w+v+e+f+2z \\
d'-l
\end{array}\hspace*{-0.1cm}\right]v^{l(r'+t-j)}\Bigg\}$\vskip2mm$
\times e_{2}^{(h-q')}e_{1}^{(k-p)}e_{2}^{(j-r')}\xi_{(-s,-t)}\hspace*{-0.1cm}\otimes\hspace*{-0.1cm}
f_{2}^{(u-d'-t'')}f_{1}^{(v-p)}f_{2}^{(w-r'+d'-q'+t'')}\eta_{(a,b)}\Bigg\}.$

\vskip3mm Using Lemma 4.1 (b), we see that \vskip3mm$\sum\limits_{0\leq l\leq d'}(-1)^{l}\hspace*{-0.1cm}\left[\hspace*{-0.2cm}\begin{array}{c}v-u-2w+b+j-t+2z+f+e+l-1 \\
l
\end{array}\hspace*{-0.2cm}\right]$\vskip2mm$\times\left[\hspace*{-0.1cm}\begin{array}{c}b+r-u-2w+v+e+f+2z \\
d'-l
\end{array}\hspace*{-0.1cm}\right]v^{l(r'+t-j)}$\vskip2mm$=\left[\hspace*{-0.1cm}\begin{array}{c}r'+t-j \\
d'
\end{array}\hspace*{-0.1cm}\right]v^{-(v-u-2w+b+j-t+2z+f+e)d'}.$

\vskip3mm So we have \vskip3mm$C=\sum\limits_{\substack{0\leq e\leq
u\\0\leq z\leq w\\0\leq e,z,~ e+z\leq
h}}\hspace*{-0.1cm}(-1)^{e+z}\left[\hspace*{-0.2cm}\begin{array}{c}b-w+j-t+j+h-k+z-1 \\
z
\end{array}\hspace*{-0.2cm}\right]$\vskip2mm$\times
\left[\hspace*{-0.2cm}\begin{array}{c}v-u-2w+b+2j-t+h-k+z+e-1 \\
e
\end{array}\hspace*{-0.2cm}\right]$\vskip2mm
$\times\Bigg\{\sum\limits_{\substack{0\leq p \leq k,v
\\0\leq r' \leq j,~0\leq d' \leq u-e,r'\\ e+z\leq q'\leq h,~e\leq
t''\leq
q'-z,u-d'}}\hspace*{-0.1cm}\left[\hspace*{-0.2cm}\begin{array}{c}b+q'-u-2w+v-p+2r'-d'+z \\
t''-e
\end{array}\hspace*{-0.2cm}\right]$\vskip2mm$\times v^{r'(j-r'-t)+p(k-p-s-j+r')+(q'-e-z)(h-q'-t-k+p+2j-2r')}
$\vskip2mm$\times v^{-(v-u-2w+b+j-t+2z+e)d'}\left[\hspace*{-0.1cm}\begin{array}{c}r'+t-j \\
d'
\end{array}\hspace*{-0.1cm}\right]
$\vskip2mm$\times\left[\hspace*{-0.1cm}\begin{array}{c}a+p-v+w-r'+d'-z \\
p
\end{array}\hspace*{-0.1cm}\right]\hspace*{-0.15cm}\left[\hspace*{-0.15cm}\begin{array}{c}b+q'-t''-w+r'-d' \\
q'-t''-z
\end{array}\hspace*{-0.15cm}\right]$\vskip2mm$\times \Bigg\{\hspace*{-0.15cm}\sum\limits_{0\leq f\leq
r'-d'}(-1)^{f}\left[\hspace*{-0.2cm}\begin{array}{c}b-w+j-t+z+f-1 \\
f
\end{array}\hspace*{-0.2cm}\right]\hspace*{-0.15cm}\left[\hspace*{-0.2cm}\begin{array}{c}b+r'-d'-w+z \\
r'-d'-f
\end{array}\hspace*{-0.2cm}\right]v^{f(r'-d'+t-j)}\Bigg\}$\vskip2mm$
\times e_{2}^{(h-q')}e_{1}^{(k-p)}e_{2}^{(j-r')}\xi_{(-s,-t)}\hspace*{-0.1cm}\otimes\hspace*{-0.1cm}
f_{2}^{(u-d'-t'')}f_{1}^{(v-p)}f_{2}^{(w-r'+d'-q'+t'')}\eta_{(a,b)}\Bigg\}.$

\vskip3mm Using Lemma 4.1 (b), we see
that\vskip3mm$\sum\limits_{0\leq f\leq
r'-d'}(-1)^{f}\hspace*{-0.15cm}\left[\hspace*{-0.2cm}\begin{array}{c}b-w+j-t+z+f-1 \\
f
\end{array}\hspace*{-0.2cm}\right]\hspace*{-0.15cm}\left[\hspace*{-0.2cm}\begin{array}{c}b+r'-d'-w+z \\
r'-d'-f
\end{array}\hspace*{-0.2cm}\right]\hspace*{-0.15cm}v^{f(r'-d'+t-j)}$\vskip2mm$
=\left[\hspace*{-0.2cm}\begin{array}{c}r'-d'+t-j \\
r'-d'
\end{array}\hspace*{-0.2cm}\right]v^{-(b-w+j-t+z)(r'-d')}.$

\vskip3mm So we have \vskip3mm$C=\sum\limits_{0\leq z\leq
w,h}\hspace*{-0.1cm}(-1)^{z}\left[\hspace*{-0.2cm}\begin{array}{c}b-w+j-t+j+h-k+z-1 \\
z
\end{array}\hspace*{-0.2cm}\right]$
\vskip2mm
$\times\Bigg\{\sum\limits_{\substack{0\leq
p \leq k,v
\\0\leq r' \leq j,~0\leq d' \leq u,r'\\ z\leq q'\leq h,~0\leq
t''\leq
q'-z,u-d'}}\hspace*{-0.1cm}v^{r'(j-r'-t)+p(k-p-s-j+r')+(q'-z)(h-q'-t-k+p+2j-2r')}$
\vskip2mm$\times
v^{-(v-u-2w+b+j-t+2z)d'-(b-w+j-t+z)(r'-d')}\left[\hspace*{-0.1cm}\begin{array}{c}r'+t-j \\
d'
\end{array}\hspace*{-0.1cm}\right]\hspace*{-0.15cm}\left[\hspace*{-0.2cm}\begin{array}{c}r'-d'+t-j \\
r'-d'
\end{array}\hspace*{-0.2cm}\right]
$\vskip2mm$\times\left[\hspace*{-0.1cm}\begin{array}{c}a+p-v+w-r'+d'-z \\
p
\end{array}\hspace*{-0.1cm}\right]\hspace*{-0.15cm}\left[\hspace*{-0.15cm}\begin{array}{c}b+q'-t''-w+r'-d' \\
q'-t''-z
\end{array}\hspace*{-0.15cm}\right]$
\vskip2mm$\times\Bigg\{\hspace*{-0.1cm}\sum\limits_{0\leq e\leq
t''}(-1)^{e}\hspace*{-0.1cm}\left[\hspace*{-0.2cm}\begin{array}{c}v-u-2w+b+2j-t+h-k+z+e-1 \\
e
\end{array}\hspace*{-0.2cm}\right]$\vskip2mm
$\times\left[\hspace*{-0.2cm}\begin{array}{c}b+q'-u-2w+v-p+2r'-d'+z \\
t''-e
\end{array}\hspace*{-0.2cm}\right]v^{e(q'-p+2r'-d'+t+k-h-2j)}\Bigg\}$
\vskip2mm$\times e_{2}^{(h-q')}e_{1}^{(k-p)}e_{2}^{(j-r')}\xi_{(-s,-t)}\hspace*{-0.1cm}\otimes\hspace*{-0.1cm}
f_{2}^{(u-d'-t'')}f_{1}^{(v-p)}f_{2}^{(w-r'+d'-q'+t'')}\eta_{(a,b)}\Bigg\}.$

\vskip3mm Using Lemma 4.1 (b), we see
that\vskip3mm$\sum\limits_{0\leq e\leq
t''}(-1)^{e}\hspace*{-0.1cm}\left[\hspace*{-0.2cm}\begin{array}{c}v-u-2w+b+2j-t+h-k+z+e-1 \\
e
\end{array}\hspace*{-0.2cm}\right]$\vskip2mm
$\times\left[\hspace*{-0.2cm}\begin{array}{c}b+q'-u-2w+v-p+2r'-d'+z \\
t''-e
\end{array}\hspace*{-0.2cm}\right]v^{e(q'-p+2r'-d'+t+k-h-2j)}$\vskip2mm$
=\left[\hspace*{-0.2cm}\begin{array}{c}q'-p+2r'-d'+t+k-h-2j \\
t''
\end{array}\hspace*{-0.2cm}\right]v^{-(v-u-2w+b+2j-t+h-k+z)t''}.$\vskip3mm

So we have\vskip3mm
$C=\sum\limits_{\substack{0\leq
p \leq k,v
\\0\leq r' \leq j,~0\leq d' \leq u,r'\\ 0\leq q'\leq h,~0\leq
t''\leq
q',u-d'}}\hspace*{-0.1cm}v^{r'(j-r'-t)+p(k-p-s-j+r')+q'(h-q'-t-k+p+2j-2r')}
$\vskip2mm$\times
v^{-(v-u-2w+b+j-t)d'-(b-w+j-t)(r'-d')-(v-u-2w+b+2j-t+h-k)t''}$
\vskip2mm$
\times\left[\hspace*{-0.1cm}\begin{array}{c}r'+t-j \\
d'
\end{array}\hspace*{-0.1cm}\right]\hspace*{-0.15cm}\left[\hspace*{-0.2cm}\begin{array}{c}r'-d'+t-j \\
r'-d'
\end{array}\hspace*{-0.2cm}\right]\hspace*{-0.15cm}\left[\hspace*{-0.2cm}\begin{array}{c}q'-p+2r'-d'+t+k-h-2j \\
t''
\end{array}\hspace*{-0.2cm}\right]$\vskip3mm
$\times\Bigg\{\sum\limits_{0\leq z\leq
q'-t''}(-1)^{z}v^{z(q'-p+t+k-h-2j+r'-d'-t'')}\hspace*{-0.05cm}
\left[\hspace*{-0.2cm}\begin{array}{c}b\hspace*{-0.03cm}-\hspace*{-0.03cm}w\hspace*{-0.03cm}+\hspace*{-0.03cm}2j\hspace*{-0.03cm}-\hspace*{-0.03cm}t\hspace*{-0.03cm}+\hspace*{-0.03cm}h\hspace*{-0.03cm}-\hspace*{-0.03cm}k\hspace*{-0.03cm}+\hspace*{-0.03cm}z\hspace*{-0.03cm}-\hspace*{-0.03cm}1 \\
z
\end{array}\hspace*{-0.2cm}\right]$\vskip3mm$\times\left[\hspace*{-0.1cm}\begin{array}{c}a+p-v+w-r'+d'-z \\
p
\end{array}\hspace*{-0.1cm}\right]\hspace*{-0.15cm}\left[\hspace*{-0.15cm}\begin{array}{c}b+q'-t''-w+r'-d' \\
q'-t''-z
\end{array}\hspace*{-0.15cm}\right]\Bigg\}$

\vskip2mm$\times
e_{2}^{(h-q')}e_{1}^{(k-p)}e_{2}^{(j-r')}\xi_{(-s,-t)}\hspace*{-0.1cm}\otimes\hspace*{-0.1cm}
f_{2}^{(u-d'-t'')}f_{1}^{(v-p)}f_{2}^{(w-r'+d'-q'+t'')}\eta_{(a,b)}.$

\vskip3mm Using Lemma 4.1 (c) (replacing $f, r, b, u, a, c$ by $z,
p, a-v+w-r'+d', q'-t'', b-w+r'-d', r'-d'+t-2j+k-h$ respectively), we
see that\vskip3mm$\sum\limits_{0\leq z\leq
q'-t''}(-1)^{z}v^{z(q'-p+t+k-h-2j+r'-d'-t'')}\hspace*{-0.05cm}
\left[\hspace*{-0.2cm}\begin{array}{c}b\hspace*{-0.03cm}-\hspace*{-0.03cm}w\hspace*{-0.03cm}+\hspace*{-0.03cm}2j\hspace*{-0.03cm}-\hspace*{-0.03cm}t\hspace*{-0.03cm}+\hspace*{-0.03cm}h\hspace*{-0.03cm}-\hspace*{-0.03cm}k\hspace*{-0.03cm}+\hspace*{-0.03cm}z\hspace*{-0.03cm}-\hspace*{-0.03cm}1 \\
z
\end{array}\hspace*{-0.2cm}\right]$\vskip3mm$\times\left[\hspace*{-0.1cm}\begin{array}{c}a+p-v+w-r'+d'-z \\
p
\end{array}\hspace*{-0.1cm}\right]\hspace*{-0.15cm}\left[\hspace*{-0.15cm}\begin{array}{c}b+q'-t''-w+r'-d' \\
q'-t''-z
\end{array}\hspace*{-0.15cm}\right]$\vskip3mm$=\sum\limits_{0\leq \delta\leq q'-t'',p}
v^{\delta(a-v+w-r'+d')+(q'-t''-\delta)(w-b-p+t-2j+k-h)}$\vskip2mm$\times
\left[\hspace*{-0.15cm}\begin{array}{c}a-v+w-r'+d'+p-q'+t'' \\
p-\delta
\end{array}\hspace*{-0.15cm}\right]\hspace*{-0.15cm}\left[\hspace*{-0.15cm}\begin{array}{c}b-w+r'-d'+q'-t'' \\
\delta
\end{array}\hspace*{-0.15cm}\right]$\vskip2mm$\times
\left[\hspace*{-0.15cm}\begin{array}{c}r'-d'+t-2j+k-h+q'-t''-\delta \\
q'-t''-\delta
\end{array}\hspace*{-0.15cm}\right].$

\vskip3mm So we have
\vskip3mm$
C=\sum\limits_{\substack{0\leq p \leq
k,v\\0\leq \delta\leq q'-t'',p
\\0\leq r' \leq j,~0\leq d' \leq u,r'\\ 0\leq q'\leq h,~0\leq
t''\leq
q',u-d'}}\hspace*{-0.1cm}v^{r'(j-r'-t)+p(k-p-s-j+r')+q'(h-q'-t-k+p+2j-2r')}$
\vskip2mm$\times v^{-(v-u-2w+b+j-t)d'-(b-w+j-t)(r'-d')-(v-u-2w+b+2j-t+h-k)t''}
$\vskip2mm$\times
v^{\delta(a-v+w-r'+d')+(q'-t''-\delta)(w-b-p+t-2j+k-h)}$
\vskip2mm$\times
\left[\hspace*{-0.1cm}\begin{array}{c}r'+t-j \\
d'
\end{array}\hspace*{-0.1cm}\right]\hspace*{-0.15cm}\left[\hspace*{-0.2cm}\begin{array}{c}r'-d'+t-j \\
r'-d'
\end{array}\hspace*{-0.2cm}\right]\hspace*{-0.15cm}\left[\hspace*{-0.2cm}\begin{array}{c}q'-p+2r'-d'+t+k-h-2j \\
t''
\end{array}\hspace*{-0.2cm}\right]$
\vskip2mm$\times
\left[\hspace*{-0.15cm}\begin{array}{c}a-v+w-r'+d'+p-q'+t'' \\
p-\delta
\end{array}\hspace*{-0.15cm}\right]\hspace*{-0.15cm}\left[\hspace*{-0.15cm}\begin{array}{c}b-w+r'-d'+q'-t'' \\
\delta
\end{array}\hspace*{-0.15cm}\right]$\vskip2mm$\times
\left[\hspace*{-0.15cm}\begin{array}{c}r'-d'+t-2j+k-h+q'-t''-\delta \\
q'-t''-\delta
\end{array}\hspace*{-0.15cm}\right]$\vskip2mm$\times
e_{2}^{(h-q')}e_{1}^{(k-p)}e_{2}^{(j-r')}\xi_{(-s,-t)}\hspace*{-0.1cm}\otimes\hspace*{-0.1cm}
f_{2}^{(u-d'-t'')}f_{1}^{(v-p)}f_{2}^{(w-r'+d'-q'+t'')}\eta_{(a,b)}.$

\vskip3mm Let $\mathbb{C}$ denote the degree of the coefficient in
the above expression $C,$ then we have \vskip2mm
$\mathbb{C}=d'(r'-d'+t-j)+(r'-d')(t-j)+t''(q'-p+2r'-d'+t+k-h-2j-t'')
-(v-u-2w+b+j-t)d'-(b-w+j-t)(r'-d')-(v-u-2w+b+2j-t+h-k)t''+
r'(j-r'-t)+p(k-p-s-j+r')+q'(h-q'-t-k+p+2j-2r')+(p-\delta)(a-v+w-r'+d'-q'+t''+\delta)+
(q'-t''-\delta)(r'-d'+t-2j+k-h)+\delta(b-w+r'-d'+q'-t''-\delta)+
\delta(a-v+w-r'+d')+(q'-t''-\delta)(w-b-p+t-2j+k-h).$

\vskip3mm Since $w\leq b+j+(j+h-k)-t, ~~v\geq a+b+j+h-s-t,~~k\geq
h+j,~~ v\geq u+w$ and also $v\geq a+w+k-s-j, ~~u+2w\leq
b+v+j-t+j+h-k\leq b+v+j-t,$ in this case we get
\vskip3mm
$\mathbb{C}\leq
-p^2+p(a-v+w+k-s-j)+pd'+(\delta-p)(q'-t''-\delta)-r'^2+d'(r'-d')+q'(-q'+p-r'-d')
+t''(q'-p+r'-t'')+\delta(b-w-t-k+2j+h)+(\delta-p)(q'-t''-\delta)\leq
p(a-v+w+k-s-j)+\delta(b-w-t-k+2j+h)\leq (\delta-p)(b-w-t-k+2j+h)\leq
0$ and $\mathbb{C}=0$ if and only if $p=r'=d'=q'=t''=\delta=0.$\vskip3mm
Meanwhile this element $C$ is fixed by the involution $\Psi$ of
$V(-s\omega_{1}-t\omega_{2})\otimes V(a\omega_{1}+b\omega_{2})$,
since the element listed in (6) is fixed by
$-:\dot{\mathbf{U}}\rightarrow \dot{\mathbf{U}}$. By using the
definitions, we can see that this element $C$ is
$(\theta_{2}^{(h)}\theta_{1}^{(k)}\theta_{2}^{(j)}\lozenge
\theta_{2}^{(u)}\theta_{1}^{(v)}\theta_{2}^{(w)})_{(s,t),(a,b)},$
$\mbox{if}~w\leq b+j+(j+h-k)-t, ~~v\geq a+b+j+h-s-t, ~~k\geq
h+j,~~ v\geq u+w.$ Hence the
element listed in (6) is an element of $\dot{\mathbf{B}}.$

\vskip3mm Similarly we can deal with (7) by using the computations in the proof of (1), repeatedly using Lemma 4.1 (b) and using Lemma 4.1 (c) once.

\vskip3mm $(8)$ For the element
$e_{2}^{(h)}e_{1}^{(k)}e_{2}^{(j)}1_{(l,m)}f_{1}^{(u)}f_{2}^{(v)}f_{1}^{(w)},$
its image under the map
$\dot{\mathbf{U}}\rightarrow V(-s\omega_{1}-t\omega_{2})\otimes
V(a\omega_{1}+b\omega_{2}),$ which is given by $u\mapsto
u(\xi_{(-s,-t)}\otimes\eta_{(a,b)}),$ is zero unless $l+2(u+w)-v=a-s$ and $m+2v-(u+w)=b-t,$ in
which case we get\vskip3mm

$e_{2}^{(h)}e_{1}^{(k)}e_{2}^{(j)}1_{(l,m)}f_{1}^{(u)}f_{2}^{(v)}f_{1}^{(w)}(\xi_{(-s,-t)}\otimes\eta_{(a,b)})$\vskip2mm$
=\sum\limits_{\substack{0\leq f \leq j\\0\leq p \leq k\\0\leq q \leq
p,u\\0\leq r \leq h}}v^{f(j-f-t)+p(k-p-s-j+f)+r(h-r-t-k+p+2j-2f)}
\hspace*{-0.1cm}\left[\hspace*{-0.15cm}\begin{array}{c}b+f-v+w \\
f
\end{array}\hspace*{-0.15cm}\right]$\vskip2mm$
\times\left[\hspace*{-0.15cm}\begin{array}{c}a+p-q-w \\
p-q
\end{array}\hspace*{-0.15cm}\right]\hspace*{-0.15cm}
\left[\hspace*{-0.15cm}\begin{array}{c}a+p-u-2w+v-f\\
q
\end{array}\hspace*{-0.15cm}\right]\hspace*{-0.15cm}\left[\hspace*{-0.15cm}\begin{array}{c}b+r-v+f+w-p+q \\
r
\end{array}\hspace*{-0.15cm}\right]$\vskip2mm$\times e_{2}^{(h-r)}e_{1}^{(k-p)}e_{2}^{(j-f)}\xi_{(-s,-t)}\otimes
f_{1}^{(u-q)}f_{2}^{(v-f-r)}f_{1}^{(w-p+q)}\eta_{(a,b)}.$\vskip3mm

Let $D$ denote the degree of the coefficient in the right-hand expression of the above identity, then we get \vskip2mm
$D=f(j-f-t)+p(k-p-s-j+f)+r(h-r-t-k+p+2j-2f)+f(b-v+w)+q(a+p-u-2w+v-f-q)+(p-q)(a-w)+r(b-v+f+w-p+q).$\vskip2mm

If $-l \geq u+k-j,~~-m\geq v-u+j,~~k\geq h+j,~~v\geq u+w$, then we
get $v\geq b+w+j-t,~~u+2w\geq a+v+k-s-j,~~k\geq h+j,~~v\geq u+w.$\vskip2mm

In this case we have $D\leq
-f^2+f(b-v+w+j-t)+p(u+w-v-p+f)+q(p+v-u-w-f-q)+r(b-v+w+j-t-r-f+q)\leq
p(-p+f)+q(p-f)-q^2-f^2-r^2-rf+qr\leq
$\vskip2mm$\begin{cases}p(q-p)+(r+f)(p-f)-qf-q^2-r^2\leq 0 & \hbox{if } p\leq f;         \\
(q-p)(p-f)-f^2-rf-(q-r)^2-qr \leq 0& \hbox{if } p > f.
\end{cases}$\vskip3mm

So we get $D\leq 0$  and  $D=0$ if and only if $p=r=f=q=0.$
Meanwhile the above expression is fixed by the involution $\Psi$ of
$V(-s\omega_{1}-t\omega_{2})\otimes V(a\omega_{1}+b\omega_{2})$,
since the element
$e_{2}^{(h)}e_{1}^{(k)}e_{2}^{(j)}1_{(l,m)}f_{1}^{(u)}f_{2}^{(v)}f_{1}^{(w)}$
is fixed by $-:\dot{\mathbf{U}}\rightarrow \dot{\mathbf{U}}$. By using
the definitions, we can see that the above expression is
$(\theta_{2}^{(h)}\theta_{1}^{(k)}\theta_{2}^{(j)}\lozenge
\theta_{1}^{(u)}\theta_{2}^{(v)}\theta_{1}^{(w)})_{(s,t),(a,b)},$ if
$v\geq b+w+j-t,~~u+2w\geq a+v+k-s-j,~~k\geq h+j,~~v\geq u+w.$ Hence the element listed in
$(8)$ is an element of $\dot{\mathbf{B}}.$
\vskip3mm
In analogy to the proof of (2), we can deal with $(9)$ and $(10)$ by using the computations in the proof of (8) and using Lemma 4.1 (b).

In analogy to the proof of $(3)$-$(5)$, we can deal with $(11)$ by using the computations in the proof of (8) and repeatedly using Lemma 4.1 (b).

In analogy to the proof of $(6)$ and $(7)$, we can deal with $(12)$ and $(13)$ by using the computations in the proof of (8),
repeatedly using Lemma 4.1 (b) and using Lemma 4.1 (c) once.
\vskip3mm
By computing the image of all these elements listed in Theorem 3.1
under the map $\dot{\mathbf{U}}\rightarrow
V(-s\omega_{1}-t\omega_{2})\otimes V(a\omega_{1}+b\omega_{2}),$ $\forall s, t, a, b\in \mathbb{N},$
which is given by $u\mapsto u(\xi_{(-s,-t)}\otimes\eta_{(a,b)}),$ we can get the corresponding elements of the canonical basis of $V(-s\omega_{1}-t\omega_{2})\otimes V(a\omega_{1}+b\omega_{2}).$

For example, as can be seen above, the element corresponding to the element listed in (1) is $(\theta_{2}^{(h)}\theta_{1}^{(k)}\theta_{2}^{(j)}\lozenge
\theta_{2}^{(u)}\theta_{1}^{(v)}\theta_{2}^{(w)})_{(s,t),(a,b)},$
$\mbox{if}~v\geq a+w+k-s-j, ~~u+2w\geq b+v+j-t, ~~k\geq h+j,~~v\geq u+w;$ the element corresponding to the element listed in (2) is $(\theta_{2}^{(h)}\theta_{1}^{(k)}\theta_{2}^{(j)}\lozenge
\theta_{2}^{(u)}\theta_{1}^{(v)}\theta_{2}^{(w)})_{(s,t),(a,b)},$
$\mbox{if}~v\geq a+w+k-s-j, ~~b+v+j-t+(j+h-k)\leq u+2w\leq b+v+j-t, ~~w\geq
b+j-t, ~~k\geq h+j,~~v\geq u+w;$ the element corresponding to the element listed in (6) is $(\theta_{2}^{(h)}\theta_{1}^{(k)}\theta_{2}^{(j)}\lozenge
\theta_{2}^{(u)}\theta_{1}^{(v)}\theta_{2}^{(w)})_{(s,t),(a,b)},$
$\mbox{if}~v\geq a+b+j+h-s-t, ~~w\leq b+j-t+(j+h-k), ~~k\geq h+j,~~v\geq u+w.$

For example, the element corresponding to the element listed in the set $(1')$ is
$(\theta_{2}^{(h)}\theta_{1}^{(k)}\theta_{2}^{(j)}\lozenge
\theta_{2}^{(u)}\theta_{1}^{(v)}\theta_{2}^{(w)})_{(s,t),(a,b)},$
$\mbox{if}~v\leq a+w+k-s-j, ~~w\geq b+j+(j+h-k)-t, ~~k\geq h+j,~~v\geq u+w;$ the element corresponding to the element listed in $(2')$ is
$(\theta_{2}^{(h)}\theta_{1}^{(k)}\theta_{2}^{(j)}\lozenge
\theta_{2}^{(u)}\theta_{1}^{(v)}\theta_{2}^{(w)})_{(s,t),(a,b)},$
$\mbox{if}~v\leq a+w+k-s-j, ~~u+2w\leq b+v+j-t+(j+h-k),~~ b+j-t+(j+h-k)\leq w\leq
b+j-t, ~~k\geq h+j,~~v\geq u+w;$ the element corresponding to the element listed in $(6')$ is
$(\theta_{2}^{(h)}\theta_{1}^{(k)}\theta_{2}^{(j)}\lozenge
\theta_{2}^{(u)}\theta_{1}^{(v)}\theta_{2}^{(w)})_{(s,t),(a,b)},$
$\mbox{if}~u+w\leq a+b+k-s-t, ~~u+2w\geq b+v+j-t, ~~k\geq h+j,~~v\geq u+w.$

All the other cases can be considered similarly. After a careful analysis, we can conclude that we exactly get all the elements of the canonical basis of
$V(-s\omega_{1}-t\omega_{2})\otimes V(a\omega_{1}+b\omega_{2}).$ By definition, these elements listed in Theorem 3.1 are complete.

Now we have completed the proof of Theorem 3.1.
\end{proof}

\vskip3mm Acknowledgements. The author would like to thank Professor
G. Lusztig both for posing this problem that led to this paper, and
for valuable guide during his stay in Beijing. The author would also
like to thank Professor Nanhua Xi for very helpful
comments.



\end{document}